\documentclass[english]{article}
\usepackage{float}
\usepackage{algorithm,algorithmic}
\usepackage{amsmath}
\usepackage{amssymb}
\usepackage{graphicx}
\usepackage{bm}
\usepackage{amsthm}
\usepackage{subcaption}
\usepackage{enumitem}
\usepackage[title]{appendix}

\usepackage{geometry}
\newgeometry{
	textheight=9in,
	textwidth=5.5in,
	top=1in,
	bottom = 1.1in,
	headheight=12pt,
	headsep=25pt,
	footskip=30pt,
	lmargin=.8in,rmargin=.8in
}

\usepackage[utf8]{inputenc} 
\usepackage[T1]{fontenc}    
\usepackage{url}            
\usepackage{booktabs}       
\usepackage{amsfonts}       
\usepackage{nicefrac}       
\usepackage{microtype}  
\usepackage{mathrsfs}

\numberwithin{equation}{section}
\usepackage{hyperref}
\hypersetup{
	colorlinks=true,
	linkcolor=red,
	filecolor=magenta,      
	citecolor=blue,
}
\newtheorem{theo}{Theorem}
\newtheorem{lem}{Lemma}
\newtheorem{prop}{Proposition}

\newtheorem{defi}{Definition}

\newtheorem{claim}{Claim}
\newtheorem{example}{Example}

\numberwithin{theo}{section}
\numberwithin{lem}{section}
\numberwithin{prop}{section}
\numberwithin{alg}{section}
\numberwithin{assum}{section}
\numberwithin{defi}{section}
\numberwithin{coro}{section}
\numberwithin{claim}{section}
\numberwithin{example}{section}

\theoremstyle{remark}
\newtheorem{remark}{Remarks}
\numberwithin{remark}{section}

\usepackage{stfloats}
\usepackage{wrapfig}

\usepackage{babel}

\usepackage{mathtools}
\def\Tiny{\fontsize{4pt}{4pt}\selectfont}
\newcommand*{\eqdef}{\ensuremath{\overset{\mathclap{\text{\Tiny def}}}{=}}}

\def\Tiny{\fontsize{4pt}{4pt}\selectfont}

\title{Equilibrate Parametrization:  Optimal   Metric  Selection  with 
	 Provable One-iteration Convergence for $ l_1 $-minimization  
 } 
	
\author{ 	
	Yifan Ran
}
\date{}
\begin{document}
	\maketitle
	\makeatletter{\renewcommand*{\@makefnmark}{}
\footnotetext{
The	author is with the department of Electrical and Electronics Engineering, Imperial College London, UK (e-mail: y.ran18@imperial.ac.uk).  \makeatother}}

\begin{abstract}
	 Incorporating a non-Euclidean variable metric to  first-order algorithms is known to bring enhancement. 
	 However, due to the lack of an optimal choice,  such an enhancement appears significantly underestimated.	
	 In this work, we establish a   metric  selection principle via optimizing a convergence rate upper-bound.  
	 For  general $ l_1 $-minimization, we  propose  an  optimal  metric choice with  closed-form  expressions guaranteed. Equipping such  a variable metric, we prove that  the  optimal solution to the $ l_1 $ problem will be obtained via a one-time proximal operator evaluation.
	 Our technique applies to a large  class of fixed-point algorithms,  particularly the ADMM,  which is  popular, general,  and requires minimum assumptions.

	The key to  our success  is the employment of an unscaled/\textit{equilibrate}  upper-bound. We show that  there exists an implicit  scaling that poses a  hidden obstacle to optimizing parameters. This turns out to be a  fundamental issue induced by the classical  parametrization. We  note that the conventional way always associates the parameter to the range of a function/operator.  This  turns out not a natural  way,   causing  certain symmetry losses, definition  inconsistencies,  and unnecessary complications, with the well-known Moreau identity being the best example.   
	We propose  \textit{equilibrate parametrization}, which associates the parameter to the domain of a function, and to both the domain and range of a  monotone operator. A series  of powerful  results are  obtained owing to the new parametrization.
	Quite remarkably, the preconditioning  technique can be  shown as equivalent to the metric selection issue.

\end{abstract}

\vspace{5pt}
\noindent 
$ \textbf{Keywords:} $\,  
 	Monotone operators, 
 	Fixed-point theory, 
 	 Variable metric  methods,
 	Optimal preconditioning,
 	Fenchel dual,
 	{Attouch–Th\'era} dual,
 	Self-duality,
 	Proximal operator,
 	Optimal parameter selection,
 	Convergence rate,
	Alternating  Direction Method of Multipliers (ADMM),
	Douglas-Rachford splitting (DRS),
	 $ l_1 $-regularization

\section{Introduction}
Convex programming \cite{rockafellar1997convex,boyd2004convex} is a powerful and well-established technique that  finds numerous applications across different fields.
Despite  the great power  and numerous benefits of a convex program, solving it often requires a significant amount of computational power. 
For practical problems, the  significance of  this issue is dramatically  amplified   due to their large-scale nature.  
How  to efficiently solve  an  optimization problem  is therefore  of  central importance.
Over the  years, intensive efforts have  been devoted to developing various algorithms.  
Among all candidates, one of the most promising and popular choices is the first-order algorithms.
As its name indicates, only first-order information (gradient or sub-gradient) is used. This contributes to  a  low per-iteration cost and  low memory storage, see a comprehensive survey by Amir Beck \cite{beck2017first}, and a simplified  view  by Marc Teboulle \cite{teboulle2018simplified}.

The first-order algorithms are a large family, and  a particularly popular one being the Alternating  Direction Method of Multipliers (ADMM) \cite{boyd12,glowinski1975approximation,gabay1976dual}.   
It receives  an increasing  amount of attention,  owing to  its mild convergence requirement,  outstanding practical performance, and  a natural structure to allow distributed/parallel optimizations.  
Specifically, unlike the  gradient-based methods where a  Lipschitz type of assumption is often needed to guarantee convergence, the requirement for ADMM  is as basic as convexity.   
In fact, even for non-convex problems, there  exist some  empirical successes, see a discussion in \cite{boyd12} and reference therein.  
The superior practical performance of  the ADMM has been driving fundamental changes.    
A series of works has been devoted to replacing   the classical general solver, originally built in other approaches such as the interior-point methods \cite{karmarkar1984new, nesterov1994interior}, into ADMM.
Some remarkable successes in building such a general ADMM-based  solver  include:  Stellato and  Boyd, et.al.   \cite{stellato2020osqp}  for quadratic problems,
O’Donoghue and Boyd, et. al. \cite{o2013operator}  for conic optimization, Zaiwen Wen et al. \cite{wen2010alternating} for  semidefinite programming.

During the development of  first-order algorithms, the proximal  operator \cite{moreau1965proximite, proxi_algs}  is shown to be an ideal analysis tool \cite{teboulle2018simplified}. It is  strongly  related to the monotone operators  \cite{bauschke2017convex, ryu2022large, ryu2016primer}, and enjoys the fixed-point theory there.
It is a standard and elegant tool that  connects  many seemingly  unrelated   algorithms.
For example, while ADMM was independently proposed  by Glowinski and Marrocco \cite{glowinski1975approximation} in 1975, and Gabay and Mercier \cite{gabay1976dual} in 1976, it was later shown as  equivalent to many other algorithms.  
A particularly important one is the Douglas-Rachford splitting (DRS) \cite{douglas1956numerical} from numerical analysis. In \cite{eckstein1989splitting}, Eckstein shows the equivalence of ADMM and DRS, via the Fenchel dual problem. In fact, via an infimal post-composition  technique, one can show  it directly from the primal problem, see \cite[sec. 3.1]{ryu2022large}. 
Another closely related one is the Peaceman-Rachford splitting (PRS) \cite{peaceman1955numerical}, which can be viewed as the  DRS/ADMM algorithm without an averaging step, see \cite{eckstein1989splitting}. 
Recently, a popular algorithm from image processing, 
the   Primal-Dual Hybrid Gradient (PDHG) method \cite{chambolle2011first, esser2010general, pock2009algorithm} turned  out equivalent to ADMM, proved by  O'Connor and  Vandenberghe \cite{o2020equivalence}. Additionally, there exists a series  of strong connections  between ADMM and Spingarn’s method of partial inverses, Dykstra’s alternating projections method, Bregman iterative algorithms for $ l_1 $ problems in signal	processing,  see more details from the review paper by  Boyd  et  al. \cite{boyd12}.  
The fact that ADMM is intrinsically connected to so many different  algorithms across various fields implies  an appealing underlying structure.


Despite the  many advantages, like   all other first-order algorithms, ADMM inherits a common weakness.
It is well-known that the performance of  first-order algorithms largely depends  on the conditioning of the problem data. 
A common heuristic alleviating this issue is  called \textit{preconditioning} \cite{nocedal1999numerical,benzi2002preconditioning, chen2005matrix, fougner2018parameter}. 
It aims to reduce the  iteration number complexity. A typical heuristic is  by appealing to  
the so-called  \textit{condition number} $\kappa  =  \sigma_{max}/\sigma_{min}$, the ratio of the largest and smallest eigenvalues.
Specifically, consider a linear system $ \bm{Ax} = \bm{b} $,  which can be rewritten  into $ \bm{EAx} = \bm{Eb} $, with $ \bm{E} $ being a full-rank square matrix.
One would select preconditioner $ \bm{E} $ such that the new matrix $ \bm{EA} $ has a  better condition number than the original one, see e.g. \cite{stellato2020osqp,giselsson2014diagonal,fougner2018parameter}. 
Ideally, we would like to attain an  optimal  condition number.  In \cite{boyd1994linear}, Boyd et al.   show that one can find it by solving a  semidefinite program (SDP), except  it  adds significant computational cost and hence is  not  practically useful.

While such a condition-number approach is perhaps the most popular one, there is no theoretical guarantee that  the iteration-number  complexity or convergence rate will necessarily  improve.  
A natural question  is  --- can we select a preconditioner by   optimizing  the convergence rate (rather than the condition number)?
Unfortunately,
as pointed out in  a recent paper by Bartolomeo Stellato and Stephen Boyd  et al. \cite[sec. 1.2, First-order methods]{stellato2020osqp} `\textit{Despite some recent theoretical results \cite{banjac2018tight,giselsson2016linear}, it remains unclear how to select those parameters to optimize the algorithm convergence rate}'.

Fortunately,  a most  recent breakthrough  on  this issue was  made  by  Yifan Ran \cite{ran2023general}, where  a scalar parameter that optimizes a convergence rate bound is found  for  ADMM. 
 It is appealing due to a closed-form expression always exists, regardless of the initialization choice. 
Our optimal metric result can be viewed as an extension   and meanwhile much more powerful.  
On one hand, this is not surprising, as   the enhancement by employing  a variable metric  has  long been  realized, see e.g. \cite{ fletcher1970new, powell1983variable, davidon1991variable, bonnans1995family, combettes2014variable}. 
On the other hand, the degree of enhancement appears significantly  underestimated in the literature. To interpret its power,  we may view employing a variable metric to the Euclidean space as introducing  curvature.  That said, the original `straight-line' step-size now becomes curved. 
Then, in principle, regardless of the initialization point and  initial direction,  an  optimal curved step-size should yield one-iteration convergence. Indeed, we are  able  to  prove such an ultimately powerful case for $ l_1 $-minimization and, quite remarkably, 
the optimal choice is in a closed-form expression.  While it is of great  interest, at this stage, it is not yet ready for practical use, due to part of optimal point information (an element-wise ratio  knowledge) is needed. The successive estimation technique from \cite{ran2023general} appears not to work well  when applied element-wisely. How to perform a good estimation such  that near  one-iteration convergence can be attained for  practical use is left for future research.




Our technique  is in a simple manner, minimizing an unscaled/\textit{equilibrate} upper-bound. The novelty is an  unscaling process and the rest is standard.  
We reveal that there exists an extra, implicit, parameter-dependent scaling  in the bound, which does not affect  the convergence property but  poses an obstacle  when  optimizing  the parameter.
Mathematically  speaking, we show  that a (firmly) non-expansive monotone operator admits a  class of scaled variants,  yielding  parallel converging sequences.  
This tricky issue  is perhaps the main reason that the  optimal parameter is unknown for algorithms such as ADMM (48-year open-problem) and  DRS (67-year open-problem). 
It turns out the extra scaling is induced by the classical  parametrization itself, due to that the parameter is  associated with the range of a function/operator.  While it is intuitively natural to  do so, it is not mathematically a good idea.
It is the cause of some inconsistent definitions in the  proximal operator, also   symmetry loss and unnecessary complications in the well-known Moreau identity.
We propose to identify and remove the extra scaling by appealing to the symmetric counterpart of an operator/function, such as  the inverse  or  adjoint. This general idea has an available interpretation as the duality.    
Alternatively,  we propose the  \textit{equilibrate parametrization}  that naturally avoids this issue.
It associates the parameter to the domain of a function, and to both the domain  and range of a  monotone operator.

Apart from the  above basic  benefits, there is more to  the \textit{equilibrate parametrization}. When we apply it to the ADMM, we obtain a direct, unified analysis framework. 
This  allows a direct fixed-point analysis on ADMM rather than the conventional way of converging to DRS, see  \cite[sec. 3.1, 3.2]{ryu2022large},  \cite{poon2019trajectory}. 
 When we  apply the \textit{equilibrate parametrization}  to the convex program itself, we reveal   a  unified  and completely symmetric primal-dual form. 
 Moreover, during the derivations, we find that the  variable metric plays   exactly the same role as a preconditioner,  leading to an interesting  conclusion  ---  {Preconditioning} is equivalent to the metric selection issue.  
 That is, the optimal variable metric obtained in this work will  also be the optimal preconditioner. As a consequence, the open problem of  optimal preconditioning is solved, see e.g.  O’Donoghue and  Boyd et al. `\textit{whether there is an optimal scaling remains open}' \cite[sec. 5]{o2016conic}.


On  the application side, we consider the class of $ l_1 $-norm regularized problems. It is of great interest  and receives intensive  studies owing to a great breakthrough,   \textit{compressed sensing}, 
see a systematic  study of the issue  by  Donoho, Candes,  Romberg, Tao, et. al.  \cite{donoho2006compressed,donoho2006most,candes2006robust,candes2005decoding,candes2006stable,candes2006near,chen2001atomic}.  Also, there exist huge efforts to  finding efficient algorithms, see e.g.  Tropp, Gilbert, Needell, Dai, Yin et al. \cite{tropp2007signal, needell2009cosamp, dai2009subspace,yin2008bregman}. Such $ l_1 $-minimization finds numerous practical uses, such as  computer vision \cite{chen2006total,yang2010review},    image inpainting  \cite{elad2005simultaneous}, compressive imaging \cite{gan2007block}, MRI \cite{lustig2008compressed}, and many others, see references from a  review  paper  \cite{rani2018systematic}.

%
%
%
%

%
%
%
%
%
%
%
%
%

\subsection{Convergence rate optimization}
Here,  we adopt a convergence  rate result  from Ernest K. Ryu and Wotao Yin \cite{ryu2022large} because it is self-contained  and generally applicable to any converging  fixed-point algorithm. We will minimize its upper-bound, owing  to their generality,  our results will  also  be generally applicable.  
\begin{lem}\cite[Theorem 1]{ryu2022large} \label{lem_rate}
	Assume $\mathcal{F}: \mathbb{R}^n \rightarrow \mathbb{R}^n $ is $\theta$-averaged with $ \theta \in\, ]0,1[ $ and  $\text{Fix}\,\, \mathcal{F} \neq \emptyset $. Then, $ \bm{\zeta}^{k+1} = \mathcal{F}{\bm{\zeta}^{k}} $ with any starting point $ \bm{\zeta}^0 \in \mathbb{R}^n $ converges to one fixed-point, i.e., 
	\begin{equation}
	\bm{\zeta}^k \rightarrow \bm{\zeta}^\star
	\end{equation}
	for some $ \bm{\zeta}^\star \in \text{Fix}\,\, \mathcal{F}$. The quantities $ \text{dist}\,\, (\bm{\zeta}^k , \text{Fix}\,\, \mathcal{F}) $, $ \Vert \bm{\zeta}^{k+1} - \bm{\zeta}^k  \Vert $, and $ \Vert \bm{\zeta}^k  - \bm{\zeta}^\star \Vert $ for any $ \bm{\zeta}^\star \in \text{Fix}\,\, \mathcal{F}$ are monotonically non-increasing with $ k $. Finally, we have
	\begin{equation}
	\text{dist}\,\, (\bm{\zeta}^k , \text{Fix}\,\, \mathcal{F}) \rightarrow 0,
	\end{equation}
	and
	\begin{equation}\label{rate}
	\Vert \bm{\zeta}^{k+1} - \bm{\zeta}^k  \Vert^2 \leq \frac{\theta}{(k+1)(1-\theta)}	\text{dist}^2\,\, (\bm{\zeta}^0 , \text{Fix}\,\, \mathcal{F}).
	\end{equation}
\end{lem}
The above  notion  `dist' denotes a best upper-bound, i.e., an abstract optimization is performed,
\begin{equation}
\text{dist} (x, \Omega) = \underset{z \in \Omega}{inf} \,\, \Vert z  - x  \Vert^2.
\end{equation} 
Intuitively, such a   best upper-bound can be explicitly found by invoking  specific fixed-point definitions. This would  reveal  the optimal parameter  choice.
For example,  a well-known ADMM  fixed-point  is  $ \gamma\bm{Ax}^\star  + \bm{\lambda}^\star $, with $ \gamma > 0 $ denoting   the step-size and
$ \cdot^\star  $ denotes the optimal solution and is fixed given a convex program. It appears that the optimal choice $ \gamma^\star $ should be easily determined.
However, later we will show that such a direct attempt would  fail. Indeed,  as pointed out in \cite{ryu2022large}  `\textit{The general question of how to optimally choose the scalar   $\gamma$ is open}'.

Hence, we may delay the optimization and start with  the original case:
\begin{equation}\label{rate00}
\Vert \bm{\zeta}^{k+1} - \bm{\zeta}^k  \Vert^2 \leq \frac{\theta}{(k+1)(1-\theta)} \Vert \bm{\zeta}^\star - \bm{\zeta}^0 \Vert^2,
\end{equation}
where $ \bm{\zeta}^\star \in \text{Fix}\,\, \mathcal{F} $ is a fixed-point. Let us note that $ \theta $ is fixed  given a certain  algorithm, and $ k $ being the iteration number counter is clearly independent of an algorithm parameter. That is, essentially only the norm term  is relevant,  hence  the topic here can also be viewed as a fixed-point selection issue.

\subsection{ADMM algorithm}
The ADMM algorithm solves the following general convex problem:
\begin{align}\label{ADMMpro}
&\underset{\bm{x},\bm{z}}{\text{minimize}}\quad   f(\bm{x}) + g(\bm{z})  \nonumber\\
&\text{subject\,to}\quad  \bm{A}\bm{x} - \bm{B}\bm{z} = \bm{c} ,
\end{align}
with variables $ \bm{x} \in \mathbb{R}^n $, $ \bm{z} \in \mathbb{R}^m $, where $ \bm{A}\in \mathbb{R}^{p \times n}  $, $ \bm{B}\in \mathbb{R}^{p \times m}  $, $ \bm{c}\in \mathbb{R}^{p}  $. We assume 
functions $ f, g $ are closed, convex, and proper (CCP), and matrices $ \bm{A} $ and $ \bm{B} $ have full column rank. 	

\subsubsection{Iterates}
The ADMM iterates are based on the following augmented Lagrangian: 
\begin{equation}\label{L_gamma}
\mathcal{L}_\gamma(\bm{x},\bm{z},\bm{\lambda}) 
\,\eqdef\,  f(\bm{x}) + g (\bm{z})  + \frac{\gamma}{2}\Vert\bm{A}\bm{x} - \bm{B}\bm{z} - \bm{c} + \bm{\lambda}/\gamma\Vert^2,  
\end{equation}	
where $\gamma > 0$ denotes a positive step-size.	 The ADMM iterates are 
\begin{align}\label{admm_iterates}
\bm{x}^{k+1} =\,\,& \underset{\bm{x}}{\text{argmin}} \,\, 	\mathcal{L}_\gamma(\bm{x},\bm{z}^k,\bm{\lambda}^k)  \nonumber\\				
\bm{z}^{k+1} =\,\, & \underset{\bm{z}}{\text{argmin}} \,\, 	\mathcal{L}_\gamma(\bm{x}^{k+1},\bm{z},\bm{\lambda}^k) \nonumber\\
\bm{\lambda}^{k+1} =\,\,  &\bm{\lambda}^{k} + \gamma( \bm{A}\bm{x}^{k+1} - \bm{B}\bm{z}^{k+1} - \bm{c}). 	
\end{align}


\subsubsection{Preconditioning}
Here, we  introduce the  preconditioning technique. We consider the  general abstract  ADMM type of formulation as in \eqref{ADMMpro}. 
Preconditioning  can be written as  
\begin{align}\label{precon}
&\underset{\bm{x},\bm{z}}{\text{minimize}}\quad  \gamma \bigg(  f(\bm{x}) +  g(\bm{z}) \bigg) \nonumber\\
&\text{subject\,to}\quad  \bm{E}(\bm{A}\bm{x} - \bm{B}\bm{z}) = \bm{E}\bm{c} ,
\end{align}
with data scaling factor $\gamma>0$  and a full-rank preconditioner $ \bm{E} $. See related work from  \cite[Sec. 5]{stellato2020osqp}, \cite[Sec. 5]{o2016conic}, \cite[Sec. 4.4.1]{fougner2018parameter}, \cite{teixeira2015admm}.

%
%
%
%
%
%
%
%
%
%


\subsection{Notations}
The Euclidean space  is denoted by $  \mathscr H $   with inner product  $ \langle \cdot, \, \cdot \rangle   $ equipped, and its induced norm $ \Vert \cdot \Vert $. We denote  its  extension to a metric space as $  \mathscr {H}_{\mathcal M} $,
with norm $\Vert\bm{v} \Vert_{\mathcal{M}} = \sqrt{\langle \bm{v},  \mathcal{M}\bm{v}\rangle} $.
We denote by $  \Gamma_0 (\mathscr H) $  the  space  of  convex, closed  and proper  (CCP) functions from $ \mathscr H $  to the extended real line  $ ]-\infty, +\infty]$,
by  $ \mathfrak{B}(\mathscr{H}, \mathscr{L}) $ the space of bounded linear operators from $ \mathscr{H}$ to $ \mathscr{L} $, with domain  $\mathscr{H}$. The abbreviation  $ \mathfrak{B}(\mathscr{H}) $ corresponds to when $ \mathscr{H} = \mathscr{L} $.
At last, the uppercase bold, lowercase bold, and not bold letters are used for matrices, vectors, and scalars, respectively. The uppercase calligraphic letters, such  as $ \mathcal{A} $ are used to denote  operators.

\subsection{Organization}
This paper is organized as follows: 
To start, we briefly review some useful relations for the monotone operators. 
In Sec. \ref{sec_pre}, we introduce our \textit{equilibrate upper-bound minimization technique} and explain  why  an unscaling process is needed. 
Particularly, in  Sec. \ref{sec_cri}, by studying symmetry, we show  how to reveal  and then remove  an extra scaling.  
Following this strategy,  in Sec. \ref{sec_class}, we identify the  source  of the  issue, by appealing to the duality. It reveals that the classical parametrization itself introduces  an extra scaling. 
In Sec. \ref{sec_equ}, we propose \textit{equilibrate parametrization} that addresses a  series of issues.  To further exploit its power,
(i)  In Sec. \ref{sec_admm}, we apply it to the ADMM  algorithm. This  yields a new, simple, and unified  fixed-point  analysis framework.
 (ii)  In Sec. \ref{sec_frame}, we apply  it to the convex  program itself, obtaining a completely symmetric primal-dual problem. 
 Remarkably, the analysis there reveals the  equivalence between preconditioning and the metric selection issue. 
 At last, in Sec. \ref{sec_opt_met}, we  construct  a closed-form metric choice for the $ l_1 $-minimization that  admits a one-iteration convergence property.

\subsection{Highlight}
Here, we briefly summarize two general results, see more details in Sec. \ref{sec_opt_met}.

\vspace{5pt}

(i) For the ADMM algorithm, the general optimal choice of  metric $ \mathcal M   =  \mathcal S^* \mathcal S$ can  be found via
\begin{align}\
 \underset{\mathcal S  }{\text{minimize}}\,\,  	 \Vert \mathcal S\bm{Ax}^\star\Vert^2  +  \Vert (\mathcal{S}^*)^{-1} \bm{\lambda}^\star    \Vert^2 - 2 \langle\mathcal S\bm{A}\bm{x}^\star, \bm{\zeta}^0 \rangle -    2 \langle(\mathcal{S}^*)^{-1} \bm{\lambda}^\star,  \bm{\zeta}^0 \rangle,
\end{align}
with $ \mathcal S  \in \mathfrak B (\mathscr H) $  being bijective, where $  \bm{\zeta}^0 = \bm{Ax}^0  + \bm{\lambda}^0 $   is an arbitrary initialization and where  $ \mathcal S^* $  denotes the adjoint operator.

\vspace{5 pt}

(ii)  Consider the following $ l_1 $-minimization:
\begin{align}\label{l1_or}
\underset{\bm{x}}{\text{minimize}}\,\,\,   f(\bm{x}) + \alpha\Vert \bm{Fx} \Vert_1,
\end{align}
with  $ f \in \Gamma_0 (\mathscr H)  $  being smooth,   $\alpha > 0$ a   regularization parameter, and $ \bm{F} $ a full column-rank  matrix.
Let the optimal variable metric be defined as
\begin{equation}
\mathcal{M^\star} (\bm{v})  = \text{abs}\bigg(  {\bm{\lambda}^\star} \oslash {\bm{Fx}^\star}   \bigg) \, \odot\, \bm{v},
\end{equation}
where  by $ \oslash $ and $ \odot $ we denote the Hadamard/element-wise division, and element-wise multiplication, respectively;
by `$ \text{abs} $' the element-wise absolute value;  by $ \bm{x}^\star $ and $ \bm{\lambda}^\star $ the primal and dual optimal solutions. See extra  details from Proposition \ref{prop_matrix_step} when  zero  elements are involved.

Then,  the following proximal  evaluation: 
\begin{equation}
\bm{x}^1 = \underset{\bm{x}}{\text{argmin}} \,\, f (\bm{x} ) + \frac{1}{2}\Vert\bm{Fx} \Vert^2_{\mathcal{M}^\star},
\end{equation}
yields an optimal solution to \eqref{l1_or}.

\section{Background: monotone operators}
Here, we introduce some basic definitions of a  useful tool, the monotone operators. It is highly unifying and strongly connected to the fixed-point theory.
We refer the interested readers to \cite{bauschke2017convex} for a systematic study.

Let  $ \mathcal{A}: \mathscr H \rightarrow 2^\mathscr H$ be a set-valued operator, $ (\bm{x}, \bm{u}) $ an  ordered pair. Then, $ \mathcal{A} $ is described by its graph:
\begin{equation}
\text{gra}\,\mathcal{A} = \big\{	(\bm{x}, \bm{u}) \in 	\mathscr H \times \mathscr H \,\,\vert\,\, \bm{u} \in \mathcal{A} \bm{x}	\big\},
\end{equation}
Its inverse is defined as
\begin{equation}
\bm{x} \in   \mathcal{A} ^{-1}\bm{u}, \qquad\qquad	 \forall(\bm{x}, \bm{u}) \in 	\mathscr H \times \mathscr H,
\end{equation}
which always exists, see \cite{combettes2014variable}.

The operator $ \mathcal{A} $ is \textit{monotone} if 
\begin{equation} 
\langle \bm{u} - \bm{v} , \, \bm{x} -\bm{y}\rangle \geq 0,
\qquad  \forall (\bm{x}, \bm{u}), (\bm{y}, \bm{v}) \in  \text{gra}\,\mathcal{A}.
\end{equation}
and \textit{maximally monotone} if it is monotone and there exists no monotone operator $ \mathcal{B}:   \mathscr H \rightarrow 2^\mathscr H$ such that 
\begin{equation}
\text{gra}\,\mathcal{A} \subset \text{gra}\,\mathcal{B} \qquad \text{and} \qquad \mathcal{A} \neq \mathcal{B},
\end{equation}
see a graphical interpretation in \cite[fig.5]{combettes2014variable}.

\begin{example}[sub-differential operator]
	Given a proper function $ f: \mathscr H \rightarrow \, ]-\infty, +\infty] $, its associated  sub-differential operator $\partial f: \mathscr H \rightarrow 2^\mathscr H$ is defined as 
	\begin{equation}
	\partial f (\bm{x}) = \{  \bm{u} \in  \mathscr H \, \vert\, \langle \bm{u},  \bm{y} - \bm{x} \rangle + f(\bm{x}) \leq   f(\bm{y}) , \quad\forall  \bm{y} \in  \mathscr H\}.
	\end{equation}
	For  $ f \in \Gamma_0 (\mathscr H)$,   $\partial f$ is maximal monotone.
\end{example}

The operator $ \mathcal{A} $ is  \textit{nonexpansive} if
\begin{equation} \label{non}
\Vert \bm{u} - \bm{v} \Vert^2 \leq \Vert \bm{x} - \bm{y} \Vert^2,
\qquad\qquad \forall (\bm{x}, \bm{u}), (\bm{y}, \bm{v}) \in  \text{gra}\,\mathcal{A}.
\end{equation}
and  \textit{firmly nonexpansive} if
\begin{equation} \label{f_non}
\Vert \bm{u} - \bm{v} \Vert^2 \leq \langle \bm{x} - \bm{y}, \bm{u} - \bm{v} \rangle,
\quad\quad \forall (\bm{x}, \bm{u}), (\bm{y}, \bm{v}) \in  \text{gra}\,\mathcal{A}.
\end{equation}
Clearly, by Cauchy–Schwarz inequality, \textit{firmly nonexpansiveness}  implies \textit{nonexpansiveness}.

\subsection{Resolvent }\label{sec_res}
Let  $ \mathcal{A}: \mathscr H \rightarrow 2^\mathscr H$ be a set-valued operator, the \textit{resolvent} and  \textit{reflected resolvent} of $ \mathcal{A} $ are 
\begin{equation}
J_\mathcal{A} = ( \mathcal{I} + \mathcal{A})^{-1}, \qquad\qquad R_\mathcal{A} = 2J_\mathcal{A} - \mathcal{I},
\end{equation}
respectively, where $ \mathcal{I} $ denotes the identity operator. 
Some useful relations  are
\begin{align}\label{rel00}
J_\mathcal{A}\,\,   \text{is firmly nonexpansive} \quad 
& \iff \quad R_\mathcal{A}\,\, \text{is  nonexpansive}  \nonumber\\
& \iff \quad  \mathcal{I} - J_\mathcal{A}\,\,   \text{is firmly nonexpansive}  \nonumber\\
& \iff \quad  J_\mathcal{A}\,\,   \text{is 1/2-averaged},
\end{align}  
see \cite[Corollary 23.11]{bauschke2017convex}.


Suppose the operator $ \mathcal{A}$ is maximal monotone. Then, 
\begin{equation}\label{re_id}
J_\mathcal{A} + J_{\mathcal{A}^{-1}} = \mathcal{I}.
\end{equation}
For  $ f \in \Gamma_0 (\mathscr H)$, $  \partial f $  is  maximal monotone. Then, 
\begin{equation}\label{equ_prox}
J_{\partial f} = \textbf{Prox}_{f},  \qquad  J_{(\partial f)^{-1}} = J_{\partial f^{*}} =  \textbf{Prox}_{f^*} ,
\end{equation}
where $ f^*(\cdot) \eqdef \sup\, \langle \bm{z}, \cdot \rangle - f(\bm{z})$ is the Fenchel conjugate function. It follows that
\begin{equation}
\textbf{Prox}_{f} + \textbf{Prox}_{f^*} =  \mathcal{I},
\end{equation}
which is known as the Moreau identity.

\subsection{Parametrization}\label{sec_pa}

Let  $ \mathcal{A}: \mathscr H \rightarrow 2^\mathscr H$ be a set-valued operator, the resolvent of $ \mathcal{A} $ can be extended to include a positive parameter $\gamma\in\, ]0, +\infty]  $  and  to a  metric space environment  $  \mathscr {H}_{\mathcal{M}}$.
\begin{equation}
J_{\frac{1}{\gamma} \mathcal{A}} = ( \mathcal{I} + \frac{1}{\gamma} \mathcal{A})^{-1},  \qquad\qquad\,\, J_{\mathcal{M}^{-1} \mathcal{A}} = ( \mathcal{I} + \mathcal{M}^{-1} \mathcal{A})^{-1}.
\end{equation}


Suppose the operator $ \mathcal{A}$ is maximal monotone. Then, $ \gamma\mathcal{A} $ and  $ \mathcal{M}^{-1}\mathcal{A} $ are maximal monotone, see \cite[Lemma 3.7]{combettes2014variable}.
In this case, the following identities hold:
\begin{equation}\label{identities}
J_{\frac{1}{\gamma}\mathcal{A}} + \left(\frac{1}{\gamma}\mathcal{I}\right)\circ J_{\gamma{\mathcal{A}^{-1}}} \circ \bigg(\gamma\mathcal{I}\bigg) = \mathcal{I},
\qquad
J_{\mathcal{M}^{-1}\mathcal{A}} + \mathcal{M}^{-1}\circ J_{\mathcal{M}{\mathcal{A}^{-1}}} \circ \mathcal{M} = \mathcal{I}.
\end{equation}

\section{Equilibrate upper-bound minimization}\label{sec_pre}
In this section, we introduce our \textit{equilibrate upper-bound minimization} technique, where `equilibrate' refers to  an unscaling process  that 
restores a lost symmetry. 
We start with showing  the  failure of a direct attempt,  which  is traced  to an implicit scaling issue.
Then, we demonstrate our technique.  Particularly, we show how to reveal  and then remove the extra scaling.
For the sake of clarity and simplicity, we discuss  the scalar case  and the extension to a metric environment is straightforward by replacing the scalar with a bijective, bounded linear operator.

\subsection{Basic upper-bound minimization}
To start, we show that a direct attempt  would fail.
Consider a fixed-point sequence generated via
\begin{equation}
\bm{\psi}^{k+1} = \mathcal{F}_1{\bm{\psi}^{k}}. 
\end{equation}
Following the rate characterization from \eqref{rate00}, we have
\begin{equation}\label{sub_rate}
\Vert \bm{\psi}^{k+1} - \bm{\psi}^k  \Vert^2 \leq\frac{\theta}{(k+1)(1-\theta)}	\Vert \bm{\psi}^\star  - \bm{\psi}^0 \Vert^2.
\end{equation}
The above  upper bound is minimized by
\begin{equation}\label{pro00}
\underset{ \bm{\psi}^\star \in \text{Fix} \,\mathcal{F}_1}{\text{minimize}}\,\,    \Vert \bm{\psi}^\star  - \bm{\psi}^0 \Vert^2.
\end{equation}
To proceed, we would invoke the specific definition of  the fixed-point $ \bm{\psi}^\star $.
We will demonstrate the failure via the ADMM  fixed-point, which was shown in  \cite{eckstein1989splitting}  in 1989. This well-known version is a dual approach,  obtained
by  applying DRS to the Fenchel dual of problem \eqref{ADMMpro}.
\begin{equation}\label{fix_set}
\text{Fix} \,\,\mathcal{F}_1 \,\, = \{  \bm{\psi}^\star \,\, |\,\,  \bm{\psi}^\star =  \gamma\bm{A}\bm{x}^\star +  \bm{\lambda}^\star, \,\, \gamma > 0\} , 
\end{equation}
where  $ \bm{x}^\star $  and $  \bm{\lambda}^\star $ denote the optimal primal  and  dual solutions, respectively.  

Let us note that the  above fixed-point set is fully characterized by parameter $\gamma$, due to the solutions being  fixed given a convex program.
Then, problem \eqref{pro00} can be  specified  into
\begin{equation}\label{comp0}
\underset{ \gamma > 0 }{\text{minimize}}\,\,    \Vert \gamma\bm{A}\bm{x}^\star +  \bm{\lambda}^\star  - \bm{\psi}^0 \Vert^2.
\end{equation}
The above program is straightforward to  solve, except  the solution  will be problematic.  For example, to guarantee  the positivity constraint, there exist cases with the solution being $ \gamma^\star  \downarrow 0 $, i.e., arbitrarily  close to zero (from above). This is  against  simulation results  and practical experience.
In fact, it is not the optimal choice, in the sense that a  contradiction happens.

\subsubsection{Contradiction}\label{sec_cont}
Here,  we  show the contradiction. First, there exists another fixed-point  characterization, corresponding to the primal problem, see  \cite[sec. 3.1]{ryu2022large}:
\begin{equation}\label{fix_set2}
\text{Fix} \,\mathcal{F}_2 \,\,\, = \{  \bm{\zeta}^\star \,\, |\,\,  \bm{\zeta}^\star =  \bm{A}\bm{x}^\star +  \bm{\lambda}^\star/\gamma \} . 
\end{equation}
We may  also substitute it  to minimizing the  rate upper-bound. 
In this  case,  problem \eqref{pro00} is  specified  into
\begin{equation}\label{comp1}
\underset{ \gamma > 0 }{\text{minimize}}\,\,    \Vert \bm{A}\bm{x}^\star +  \bm{\lambda}^\star/\gamma  - \bm{\zeta}^0 \Vert^2.
\end{equation}
This  program is also easy to solve. However, similar to the previous case, we still encounter the issue $ \gamma^\star  \downarrow 0 $. 

More importantly, problems \eqref{comp1} and \eqref{comp0} admit  different solutions. 
As will be shown later in Sec.  \ref{self_duality}, the primal and dual sequences are parallel if the initializations satisfy  $  \bm{\psi}^{0} =  \gamma \bm{\zeta}^{0} $ (consider  zero-initialization for a quick check). For such  parallel sequences, their optimal convergence rates should be  the same, yet now we obtain different ones. This  implies a contradiction.

%
%
%
%
%
%
%

\subsubsection{Hidden scaling challenge}
Here, we explain why the  contradiction  happens.  Simply put, this is due to the upper-bound may  potentially  contain an extra  $\gamma$-related scaling factor, and we cannot optimize it  unless  it is removed. 

Let us note that the rate bound as in \eqref{rate00} is  scaling invariant. Given any   $\gamma > 0$,
 the following always holds:
\begin{align}\label{sc_ch}
\Vert  \bm{\psi}^{k+1} - \bm{\psi}^k  \Vert^2 \leq \frac{\theta}{(k+1)(1-\theta)}		\Vert \bm{\psi}^\star  - \bm{\psi}^0 \Vert^2
\,\,\iff\,\, \gamma \Vert  \bm{\psi}^{k+1} - \bm{\psi}^k  \Vert^2 \leq \frac{\theta}{(k+1)(1-\theta)}	\cdot \gamma	\Vert \bm{\psi}^\star  - \bm{\psi}^0 \Vert^2.
\end{align}
Clearly,  minimizing the above two upper-bounds  w.r.t. $\gamma$  will  be different.

\subsection{Solution in the literature}
Before providing our solution, we  briefly review a typical success in  the literature, which   avoids the above implicit scaling issue.
Indeed, a natural way to address 
such a scaling challenge is by division. 

Consider minimizing the following ratio:
\begin{equation}\label{fac}
\delta
= \underset{\{k | \bm{\psi}^{k+1}\neq \bm{\psi}^{k}\}}{\sup} \, \frac{\Vert \bm{\psi}^{k+1} - \bm{\psi}^\star \Vert}{\Vert \bm{\psi}^{k} - \bm{\psi}^\star \Vert}.
\end{equation}
Clearly, a positive  scaling   on both the numerator and denominator   does not change the factor $ \delta $.
This approach is proposed in \cite{ghadimi2014optimal} for quadratic programming, and  \cite{shi2014linear} for  decentralized consensus problems. 
Despite the  aforementioned successes,  there are two main drawbacks  to such a  technique:  

\vspace{5pt}

$ \bullet $ (i) First,  the factor $ \delta $ needs to be the same at every iteration $ k $, otherwise  optimizing it only improves a  few middle  steps, and  the whole convergence  rate does not  necessarily  improve.  This requirement
is satisfied under a linear rate assumption, which does not  hold in general. Theoretically, it can be achieved by imposing certain strong assumptions. However,  even  if a linear rate naturally exists, an optimization problem in such a division form is not easy to solve. Indeed, their procedures including  the  final results are very complicated.

\vspace{5pt}
$ \bullet $ (ii) More importantly, the successes in \cite{ghadimi2014optimal, shi2014linear}  require analytical tractability, i.e., the fixed-point operator $ \mathcal{F} $ (or equivalently all ADMM iterates) admits an explicit analytical form.  
This is a very strong requirement, and only available for a very limited class of problems. In view of this, it appears that there is no  hope to obtain a general optimal parameter  following this path.

\subsection{Equilibrate upper-bound}
Above, we observe some strong and seemingly unresolvable limitations from the existing  approach.
This motivates us to reconsider   upper-bound minimization. It is most promising  since it is associated  with a  much  simpler optimization problem  and  does not require any  strong assumption. Moreover, owing to Lemma \ref{lem_rate}, the obtained  result will be  generally applicable to any (converging) fixed-point  algorithm.

Again, the key challenge is  the potential extra  scaling, as indicated in  \eqref{sc_ch}.
To this end, we ask the following question:
\begin{center}
	\emph{For such  potentially scaled  upper bound \eqref{sc_ch}, can we find an unscaled one?}
\end{center}
It  turns out indeed possible.

\subsubsection{Parallel fixed-point sequences}
First, we show that the fixed-point operator itself contains  scaling variants, with the  same converging  property. That  said, there  exist parallel  fixed-point sequences.


To see  this, let us  note that 
\begin{align}
\,\,    \bm{y}^{k+1} = \mathcal{F}{\bm{y}^{k}}
\iff 
&\,\,    \bm{y}^{k+1} = \alpha^{-1}\mathcal{I}\circ \bigg(\alpha\mathcal{I}\circ  \mathcal{F}\circ \alpha^{-1}\mathcal{I}\bigg)  \circ  \alpha\mathcal{I}  \,\, \bm{y}^{k}, \qquad \alpha \neq 0 \nonumber\\
\iff 
&\,\,    \widetilde{\bm{y}}^{k+1} = \widetilde{\mathcal{F}}  \widetilde{\bm{y}}^{k},
\end{align}
with  substitutions:
\begin{equation}\label{def0}
 \widetilde{\bm{y}}^k \eqdef \alpha \bm{y}^k, \qquad \widetilde{\mathcal{F}} \eqdef \alpha\mathcal{I}\circ  \mathcal{F}\circ \alpha^{-1}\mathcal{I}.
\end{equation}

\begin{lem}
Suppose operator $ \mathcal{F} $  is  (firmly)  nonexpansive, then $ \widetilde{\mathcal{F}} \eqdef \alpha\mathcal{I}\circ  \mathcal{F}\circ \alpha^{-1}\mathcal{I} $ with $\alpha\neq  0$ is  also  (firmly)  nonexpansive.
\end{lem}
\begin{proof}
To see this, recall   the firm nonexpansiveness definition from \eqref{f_non}, we obtain
\begin{align} 
&\quad \Vert \mathcal{F}  \bm{x} -  \mathcal{F}\bm{y}\Vert^2 \leq \quad\,\,\langle \bm{x} - \bm{y}, \mathcal{F}  \bm{x} - \mathcal{F} \bm{y} \rangle, \nonumber\\
\iff
&\alpha^2 \Vert \mathcal{F}  \bm{x} -  \mathcal{F} \bm{y} \Vert^2 \leq \alpha^2\langle \bm{x} - \bm{y}, \mathcal{F}  \bm{x} - \mathcal{F} \bm{y} \rangle ,  \quad \alpha \neq 0 \nonumber\\
\iff
&\quad \Vert \widetilde{\mathcal{F}}  \widetilde{\bm{x}} -  \widetilde{\mathcal{F}}\widetilde{\bm{y}}\Vert^2 \leq \quad\,\langle \widetilde{\bm{x}} - \widetilde{\bm{y}}, \widetilde{\mathcal{F}}  \widetilde{\bm{x}} - \widetilde{\mathcal{F}} \widetilde{\bm{y}} \rangle,
\end{align}
with substitutions:
\begin{equation}
\widetilde{\bm{x}} \eqdef \alpha \bm{x}, \qquad\widetilde{\bm{y}} \eqdef \alpha \bm{y}, \qquad \widetilde{\mathcal{F}} \eqdef \alpha\mathcal{I}\circ  \mathcal{F}\circ \alpha^{-1}\mathcal{I}.
\end{equation}
The nonexpansive  operator  case shares similar arguments and  is omitted to  avoid repeating. The proof is now concluded.
\end{proof}

Following the  above lemma, we obtain  parallel converging sequences  as
\begin{equation}
 \, \bm{y}^{k} \rightarrow \bm{y}^\star \in \text{Fix}\, \mathcal{F} 
\,\,\iff\,\,
\widetilde{\bm{y}}^k \rightarrow \widetilde{\bm{y}}^\star \in \text{Fix}\, \widetilde{\mathcal{F}} .
\end{equation}
We can  characterize the class of such parallel fixed-point operator as
\begin{equation}\label{sym_set}
\{ \mathcal{T}\,\,|\,\,   \mathcal{T} = \alpha\mathcal{I}\circ  \mathcal{F}\circ \alpha^{-1}\mathcal{I},\,\,\, \alpha \neq 0\}.
\end{equation}
The insight is that  the above $\alpha$-parametrization is not symmetric. Then, if we take symmetric type of operations on $ \mathcal{T} $, it should be possible to reveal $\alpha$. Then, we can remove such an extra scaling.
This is elaborated next.

\subsubsection{Reveal the extra scaling}\label{sec_cri}
Following from the previous section, the key issue  is
\begin{center}
	\emph{How to find an unscaled operator, or how do we check and remove an extra scaling?}
\end{center}
In  view  of  the set  \eqref{sym_set}, the answer is by appealing  to  symmetry. Informally speaking, if the operator  is intrinsically symmetric, then
after a `reflection', only the unscaled one does not change. Such a `reflection' purpose can be achieved  by taking  inverse or adjoint.

Specifically,
consider a certain fixed-point operator  $ \mathcal{F}_c $. We have the following characterizations:
\begin{equation}\label{char}
\mathcal{F}_c = \alpha\mathcal{I}\circ  \mathcal{F}_0 \circ \alpha^{-1}\mathcal{I}, 
\qquad 
\mathcal{F}_c^{-1} = \alpha^{-1}\mathcal{I}\circ  \mathcal{F}_0^{-1}  \circ \alpha\mathcal{I},
\qquad 
\mathcal{F}_c^{*} = \alpha^{-1}\mathcal{I}\circ  \mathcal{F}_0^{*}  \circ \alpha\mathcal{I},
\end{equation}
where $ \mathcal{F}_0 $ denotes the underlying unscaled fixed-point operator. 
Below, we discuss the adjoint  operator $ \mathcal{F}_c^{*} $ case, similar arguments from below will hold straightforwardly for  the inverse one $ \mathcal{F}_c^{-1} $.


Consider  specific fixed-point iterates,
\begin{equation}
\bm{\zeta}_1^{k+1} = \mathcal{F}_c\bm{\zeta}_1^{k} =  \bigg( \alpha\mathcal{I}\circ  \mathcal{F}_0 \circ \alpha^{-1}\mathcal{I}\bigg) \bm{\zeta}_1^{k} ,
\qquad \bm{\zeta}_2^{k+1} = \mathcal{F}_c^* \bm{\zeta}_2^{k}  = \bigg(\alpha^{-1}\mathcal{I}\circ  \mathcal{F}_0^* \circ \alpha\mathcal{I}\bigg) \bm{\zeta}_2^{k},
\end{equation}
which gives
\begin{equation}\label{fix_rel}
\alpha^{-1}\bm{\zeta}_1^{k+1} =   \mathcal{F}_0\, \big(\alpha^{-1}\bm{\zeta}_1^{k}\big),
\qquad 
\alpha\bm{\zeta}_2^{k+1} =   \mathcal{F}_0^* \, \big(\alpha\bm{\zeta}_2^{k}\big).
\end{equation}

Now, suppose the following hold:
\begin{equation}\label{assum}
 \mathcal{F}_0 = \mathcal{F}_0^*, \qquad\quad
\alpha^{-1}\bm{\zeta}_1^{0} = \alpha\bm{\zeta}_2^{0},  \tag{assum.}
\end{equation}
Then, following from \eqref{fix_rel}, we have 
\begin{equation}\label{measure}
\alpha^{-1}\bm{\zeta}_1^{k} = \alpha\bm{\zeta}_2^{k}, \quad  \forall  k = 1, 2, \dots. \tag{criterion},
\end{equation}
which implies a fixed-point set relation
\begin{equation}\label{rel01}
\alpha^{-1}\,  \text{Fix} \,\mathcal{F}_c    \,\,=\,\, \alpha\,  \text{Fix} \,\mathcal{F}_c^*.
\end{equation}
Clearly, by simply comparing the fixed-points, we would  know exactly the value of $\alpha$.
Particularly, when  \eqref{measure}  holds with $\alpha =  1$, then 
\begin{equation}
\bm{\zeta}_1^\star  = \bm{\zeta}_2^\star \,\,\in\,\, \text{Fix} \,\mathcal{F}_0    \,\,=\,\,   \text{Fix} \,\mathcal{F}_0^*,
\end{equation}
are unscaled fixed-points.

\begin{remark}
In view of our assumption   \eqref{assum},  the first one states an intrinsic structure  relation, and  can be checked  by removing  all parametrization.  
For the  second  relation  in \eqref{assum}, it  is always achievable due to the initializations can be arbitrarily chosen, recall Lemma \ref{lem_rate}.
\end{remark}


%

\subsubsection{Contradiction revisit}
With all results established,  we may revisit the contradiction in Sec. \ref{sec_cont}.
 We observe that the ADMM primal and dual fixed-point sets as in \eqref{fix_set}, \eqref{fix_set2} are related via
\begin{equation}
\frac{1}{\sqrt{\gamma}}\text{Fix} \,\mathcal{F}_\text{1} = \sqrt{\gamma}\,\,\text{Fix} \,\mathcal{F}_\text{2}.
\end{equation}
This  would correspond to \eqref{rel01}. Also, the ADMM fixed-point operator is symmetric, a property known as self-duality (to be introduced next).  That is, the extra scaling factor $\alpha$  corresponds to $\sqrt{\gamma}$, and we can remove it to obtain the unscaled fixed-point sets, as
\begin{align}
\frac{1}{\sqrt{\gamma}}\,\,\text{Fix} \,\mathcal{F}_\text{1}
=\sqrt{\gamma}\,\,\text{Fix} \,\mathcal{F}_\text{2} 
= \{  \bm{y}^\star \, |\,  \bm{y}^\star =  {\sqrt{\gamma}}\bm{A}\bm{x}^\star +  \bm{\lambda}^\star/\sqrt{\gamma} \}.
\end{align}  
Using the above unscaled fixed-point, the previous contradiction is  clearly addressed, since the fixed-points from the primal and dual problems are now the same.

Above, we have shown how to reveal and remove the extra scaling.  Then, a natural question is what causes this phenomenon. Also, is it possible to directly obtain the unscaled fixed-point rather than the above a little bit complicated procedure? We will show that the classical  parametrization itself is the cause. These issues can be naturally avoided using our \textit{equilibrate parametrization}.

\section{Classical parametrization: implicit scaling}\label{sec_class}
In the previous section, we present our \textit{equilibrate upper-bound minimization} technique. We propose to examine and  potentially remove an extra scaling   by a `reflection' operation. 
Here, we will see that the Fenchel dual or monotone inverse is exactly the `reflection' we need. 
Also,  some similar  expressions from the previous section naturally arise here, implying the existence of an extra scaling.

%

\subsection{Duality}\label{duality}
Consider the following convex program  (primal problem):
\begin{equation}
\underset{{\bm x}}{\text{minimize}}\quad f (\bm{x})  + g (\bm{x}) ,
\end{equation}
where $ f, g \in \Gamma_0 (\mathscr H) $. Its Fenchel dual problem can be written as
\begin{equation}
\underset{{\bm \lambda}}{\text{minimize}} \quad f^* (-\bm{\lambda})  + g^* (\bm{\lambda}) ,
\end{equation}
where   $ f^*(\cdot) \eqdef \sup\, \langle \bm{z}, \cdot \rangle - f(\bm{z})$ is the Fenchel conjugate function. 
The above can be equivalently written as  monotone  inclusion problems:
\begin{align}
& \text{find} \,\,\, \bm{x}\in \mathscr H \quad \text{such that}\,\, 0 \in \mathcal{A}\bm{x} + \mathcal{B}\bm{x}\\
& \text{find} \,\, \bm{\lambda}\in \mathscr H \quad\,\text{such that}\,\, 0 \in {\mathcal{A'}}^{-1}\bm{\lambda} + \mathcal{B}^{-1}\bm{\lambda},
\end{align}
with  $ \mathcal{A} =  \partial f$ and $ \mathcal{B} =  \partial g$  being maximal monotone, where
\begin{equation}\label{at_dual}
{\mathcal{A'}}^{-1} =  \bigg(\mathcal{A}\circ (-\mathcal{I})\bigg)^{-1} =  (-\mathcal{I}) \circ \mathcal{A}^{-1}\circ (-\mathcal{I}),
\end{equation}
is   referred to as the \textit{Attouch–Th\'era} duality  \cite{attouch1996general}.

\subsection{Self-duality}\label{self_d}
Here, we discuss a  self-duality property.  If it holds, then the iteration-number complexity for solving the primal and dual problems would be the same. It is an intrinsic property for PRS and DRS/ADMM  algorithms.

Without parametrizations, we have the following characterizations:
\begin{align}
&\qquad(\text{PRS})\qquad \quad		\mathcal{F}_\text{pri} = R_\mathcal{A}R_\mathcal{B}, \qquad \qquad\qquad \mathcal{F}_\text{dual} = R_{\mathcal{A'}^{-1}}R_{\mathcal{B}^{-1}},  \nonumber\\
&(\text{DRS/ADMM})\quad  	\mathcal{F}_\text{pri} = \frac{1}{2}\mathcal{I} + \frac{1}{2} R_\mathcal{A}R_\mathcal{B}, 
\qquad\,\,\,\, \mathcal{F}_\text{dual} = \frac{1}{2}\mathcal{I} + \frac{1}{2} R_{\mathcal{A'}^{-1}}R_{\mathcal{B}^{-1}}.
\end{align}
The following relation is referred to as the \textit{self-duality}, see  \cite[Corollary 4.2, 4.3]{bauschke2012attouch},  \cite[Sec. 9.3]{ryu2022large}:
\begin{equation}\label{self}
\mathcal{F}_\text{pri} =  \mathcal{F}_\text{dual},
\end{equation}

The dual  is the `reflection'  operation we need, as discussed in the previous section.  Specifically, the first relation in \eqref{assum} now holds. Next, we will see an extra scaling arises due  to the classical way of parametrization.


%

\subsubsection{Loss  \& recovery}\label{self_duality}
Under classical parametrization,  self-duality \eqref{self} no longer holds.  The following operator relation \eqref{oper}  is well-known, traced back to \cite[Lemma 3.5 on p.125, Lemma 3.6 on p.133]{eckstein1989splitting} in  1989, but our interpretation appears new. 
\begin{align}\label{oper}
R_{\frac{1}{\gamma}\mathcal{A}} R_{\frac{1}{\gamma}\mathcal{B}} &= \bigg(\frac{1}{\gamma}\mathcal{I}\bigg)\circ R_{\gamma\mathcal{A'}^{-1}}R_{\gamma\mathcal{B}^{-1}} \circ \bigg({\gamma}\mathcal{I}\bigg) \nonumber\\
\iff \,\, 
 \frac{1}{2}\mathcal{I} + \frac{1}{2} R_{\frac{1}{\gamma}\mathcal{A}} R_{\frac{1}{\gamma}\mathcal{B}} &= \bigg(\frac{1}{\gamma}\mathcal{I}\bigg)\circ \bigg(\frac{1}{2}\mathcal{I} + \frac{1}{2} R_{\gamma\mathcal{A'}^{-1}}R_{\gamma\mathcal{B}^{-1}}\bigg) \circ \bigg({\gamma}\mathcal{I}\bigg) ,
\end{align}

We can  abstractly describe the above as
\begin{equation}
\mathcal{F}_{\frac{1}{\gamma}\text{pri}} 
= \bigg(\frac{1}{\gamma}\mathcal{I}\bigg) \circ \mathcal{F}_{\gamma\text{dual}} \circ \bigg(\gamma\mathcal{I} \bigg).
\end{equation}
Denote their iterates as
\begin{equation}\label{fix_itr}
\bm{\zeta}^{k+1} = \mathcal{F}_{\frac{1}{\gamma}\text{pri}}  \bm{\zeta}^{k}, \qquad \bm{\psi}^{k+1} = \mathcal{F}_{\gamma\text{dual}} \bm{\psi}^{k}.
\end{equation}
Then,
\begin{align}
\bm{\zeta}^{k+1} = \mathcal{F}_{\frac{1}{\gamma}\text{pri}}  \bm{\zeta}^{k} 
= \bigg(\frac{1}{\gamma}\mathcal{I}\bigg) \circ\mathcal{F}_{\gamma\text{dual}}\circ \bigg(\gamma\mathcal{I} \bigg) \bm{\zeta}^{k}
\iff
\bigg(\gamma \bm{\zeta}^{k+1}\bigg) 
= \mathcal{F}_{\gamma\text{dual}}\circ \bigg(\gamma \bm{\zeta}^{k}\bigg) .
\end{align}
Comparing it  to \eqref{fix_itr}, we obtain
\begin{equation}\label{com}
\frac{1}{\sqrt{\gamma}}\bm{\psi}^{k}  =\sqrt{\gamma} \bm{\zeta}^{k} , \,\, \forall k. 
\end{equation}
This  implies that if we  choose initializations $ \bm{\psi}^0 =  \gamma\bm{\zeta}^0 $, then \eqref{assum} holds and the two sequences  $ \{ \bm{\psi}^k \} $ and  $ \{ \bm{\zeta}^k \} $ are parallel, indicating the same convergence trajectory.
   Compare the above \eqref{com} to \eqref{measure}, we see that there exists an extra scaling  factor $ \sqrt{\gamma} $. 
   
We may remove the extra  scaling to restore self-duality. This leads to the following  unscaled fixed-point operator:
\begin{equation}\label{uns}
\mathcal{F}_0  
= \bigg(\sqrt{\gamma}\,\mathcal{I} \bigg) \circ  \mathcal{F}_{\frac{1}{\gamma}\text{pri}} \circ \bigg(\frac{1}{\sqrt{\gamma}}\,\mathcal{I}\bigg) 
= \bigg(\frac{1}{\sqrt{\gamma}}\, \mathcal{I}\bigg) \circ \mathcal{F}_{\gamma\text{dual}} \circ \bigg(\sqrt{\gamma}\,\mathcal{I} \bigg).
\end{equation}
The above unscaled form  will naturally arise in the next section using  the new parametrization.

\section{Equilibrate parametrization}\label{sec_equ}
In the previous section, we observe that the extra scaling issue is intrinsically connected with the classical way of   parametrization. Here, we will show some additional drawbacks of the classical way, and demonstrate the mathematical superiority
of  a  new  \textit{equilibrate parametrization}  technique.

\subsection{Motivations}\label{motives}
While the  extra scaling  issue from the previous section  can  be  circumvented via \eqref{uns}, there  are some other issues that are seemingly not resolvable. This  aspect shows the necessity to  consider a new parametrization.
Below,  through a fundamental tool,  the proximal operator \cite{proxi_algs}, we show that the classical parametrization technique leads to inconsistent definitions and unnecessary complications.

\vspace{8pt}

$\bullet$ (i) The  proximal operator  equipped with a classical scalar   is not perfectly defined. To see this, recall its definition as
\begin{defi}\label{prox}
	Given $ f \in \Gamma_0 (\mathscr H) $ and parameter $\gamma > 0$, the  proximal operator can be defined as
	\begin{equation}\label{class_prox}
	\textbf{Prox}_{\frac{1}{\gamma} f}(\bm{v}) \eqdef  \underset{\bm{z}}{\text{argmin}} \,\,  \frac{1}{\gamma} f(\bm{z}) + \frac{1}{2}\Vert \bm{z} - \bm{v}\Vert^2
	= \underset{\bm{z}}{\text{argmin}} \,\,   f(\bm{z}) + \frac{\gamma}{2}\Vert \bm{z} - \bm{v}\Vert^2,
	\end{equation}
\end{defi}	
The above can be rewritten   into 
\begin{equation}
\textbf{Prox}_{\frac{1}{\gamma} f}(\bm{v}) 
= \underset{\bm{z}}{\text{argmin}} \,\,   f(\bm{z}) + \frac{1}{2} \left\Vert \sqrt{\gamma}\bm{z} - \sqrt{\gamma}{\bm{v}}  \right\Vert^2.
\end{equation}
We see that the input is defined as $ \bm{v} $, while we are actually handling  $ \sqrt{\gamma}\bm{v} $, implying a definition  inconsistency.

\vspace{8pt}

$\bullet$ (ii) Another inconsistency arises in  a metric space environment $  \mathscr {H}_{\mathcal M} $.  We see that the generalized definition  does not nicely  subsume  the previous  scalar case.
\begin{defi}\label{cla_prox}
	Given $ f \in \Gamma_0 (\mathscr H) $ and a metric space environment $  \mathscr {H}_{\mathcal M} $, the proximal operator is defined as
	\begin{equation}
	\textbf{Prox}_{f}^{\mathcal{M}^{-1} } (\cdot) \eqdef  \underset{\bm{z}}{\text{argmin}} \,\, f(\bm{z}) + \frac{1}{2}\Vert \bm{z} - \cdot\Vert^2_{\mathcal{M}},
	\end{equation} 
	where $ \Vert \cdot\Vert_{\mathcal{M}} $ is  induced by
	the inner product $\langle \cdot\,,\,\cdot \rangle_\mathcal{M}  =  \langle \cdot\,,\mathcal{M} \,\cdot \rangle$.
\end{defi}
Unlike the scalar case  with two definitions, here only   one is available. This  is due to  
the composition `$ \mathcal{M}^{-1}\circ f $' is not well-defined, and hence notation `$ \textbf{Prox}_{\mathcal{M}^{-1}f} $' is typically avoided in the literature.


\vspace{8pt}

$\bullet$ (iii)  The Moreau identity suffers  from  symmetry  loss and additional complications.
To see this,  recall the non-parametrized Moreau identity:
\begin{equation}\label{md}
\bm{v} = \textbf{Prox}_{ f}(\bm{v}) + \textbf{Prox}_{f^*} (\bm{v}),
\end{equation}
Its extensions to  include a positive parameter $\gamma > 0  $  and  to a  metric space environment  $  \mathscr {H}_{\mathcal{M}}$ are
\begin{align}\label{class_more}
\bm{v} = \, \,\,\textbf{Prox}_{\frac{1}{\gamma} f}(\bm{v}) \,\,+\,\, \frac{1}{\gamma}\textbf{Prox}_{{\gamma}f^*} ({\gamma}\bm{v}),  \quad
\bm{v} = \, \textbf{Prox}_f^{\mathcal{M}^{-1}}(\bm{v}) + \mathcal{M}^{-1}\textbf{Prox}_{f^*}^\mathcal{M} (\mathcal{M}\bm{v}).
\end{align}
Comparing such extensions to the non-parametrized one \eqref{md}, we see that  symmetry is lost between the components and  the  expression is significantly complicated.

\subsection{Equilibrate proximal operator}
Following the above motivations, we present the  \textit{equilibrate parametrization} in the proximal operator case. It admits  a simple form and a nice interpretation. It  is achieved by associating the parameter to the domain of a function, rather than conventionally the range.

\begin{prop}\label{new_def1}[scalar]
	Given $ f \in \Gamma_0 (\mathscr H) $ and a scalar parameter $\rho \neq 0$, we propose the following \textit{equilibrate proximal operator}:
	\begin{equation}\label{def_right}
	\textbf{Prox}_{f \frac{1}\rho} (\bm{v}) \eqdef \underset{\bm{z}_1}{\text{argmin}} \,\, f\left(\frac{1}\rho \,\bm{z}_1 \right)+ \frac{1}{2}\Vert \bm{z}_1 - \bm{v}\Vert^2
	=\rho\,\underset{\bm{z}_2}{\text{argmin}} \,\, f(\bm{z}_2)+ \frac{1}{2}\Vert \rho{\bm{z}_2 }- \bm{v}\Vert^2.
	\end{equation}
\end{prop}
\begin{proof}
	The second equality above is not  obvious. To show it, let $ \bm{z}_1^\star =  \text{argmin}\,\, f(\bm{z}_1/\rho)$ $+ \frac{1}{2}\Vert \bm{z}_1 - \bm{v}\Vert^2 $. Define variable substitution  $ \bm{z}_2 \eqdef \bm{z}_1/\rho$. Then, 
	$ \bm{z}_2^\star = \bm{z}_1^\star/\rho $ is the minimizer to  function:  $f(\bm{z}_2)+ \frac{1}{2}\Vert  \rho\bm{z}_2 - \bm{v}\Vert^2 $, which concludes the proof. 
\end{proof}

\begin{remark}
Motivation (i)  is addressed, since the input $ \bm{v} $ is no longer affected by parameter.
\end{remark}

Next, we consider the extension  to   a metric  space. It can be  viewed as employing a decomposed metric as below.

\begin{lem}[Euclidean embedding]\label{lem_emb}
Given a metric environment $\mathscr H_\mathcal{M}$, the associated inner product can be translated into the following decomposed form:
\begin{equation}\label{metric_space}
\langle \cdot, \,  \cdot \rangle_\mathcal{M} = \langle \cdot, \,  \mathcal{M}\,\cdot \rangle = \langle \mathcal{S}\,\cdot, \,  \mathcal{S}\,\cdot \rangle ,
\end{equation}
with $ \mathcal S  \in \mathfrak B (\mathscr H) $ being bijective, which always exists due to metric $ \mathcal{M} $ is positive definite.
\end{lem}

\begin{prop}\label{new_def2}[metric space]
	Given $ f \in \Gamma_0 (\mathscr H) $ and $ \mathcal S  \in \mathfrak B (\mathscr H) $ being bijective, we define the extended \textit{equilibrate proximal operator} as:
	\begin{equation}\label{def_new}
	\textbf{Prox}_{f\mathcal{S}^{-1}} (\bm{v}) \eqdef \underset{\bm{z}_1}{\text{argmin}} \,\, f\left(\mathcal{S}^{-1}\bm{z}_1\right)+ \frac{1}{2}\Vert \bm{z}_1- \bm{v}\Vert^2
	=  \mathcal{S} \, \underset{\bm{z}_2}{\text{argmin}} \,\, f(\bm{z}_2)+ \frac{1}{2}\Vert \mathcal{S}\bm{z}_2 - \bm{v}\Vert^2.
	\end{equation}
\end{prop}

\begin{remark}
	Motivation (ii)   is addressed, since the above metric form is consistent with the scalar case. 
\end{remark}

\subsubsection{Equilibrate Moreau identity}
Here, we address  motivation (iii).
For  the sake of generality, we prove  the  metric environment case, which subsumes  the scalar one.
\begin{prop}\label{pmoreau}
	Given $ f \in \Gamma_0 (\mathscr H) $ and  $ \mathcal S  \in \mathfrak B (\mathscr H) $ being bijective, the following relation holds:
	\begin{equation}\label{Moreau identity}
\bm{v}  =\,\textbf{Prox}_{f\mathcal{S}^{-1}} (\bm{v}) +  \textbf{Prox}_{f^*{\mathcal{S}^*}} (\bm{v}),
	\end{equation}
\end{prop}
\begin{proof}
	To start, define the following variable substitutions:
	\begin{align}\label{defs}
	\bm{x} & = \underset{\bm{z}}{\text{argmin}} \,\, f(\bm{z})+ \frac{1}{2}\Vert \mathcal{S}\bm{z} - \bm{v}\Vert^2 = \mathcal{S}^{-1} \circ\textbf{Prox}_{f\mathcal{S}^{-1}} (\bm{v}) \nonumber\\
	\bm{y} & = \underset{\bm{z}}{\text{argmin}} \,\, f^*(\bm{z})+ \frac{1}{2}\Vert (\mathcal{S}^*)^{-1}\bm{z} - \bm{v}\Vert^2 =  \mathcal{S}^*\circ \textbf{Prox}_{f^*{\mathcal{S}^*}} (\bm{v}).
	\end{align}
	  The identity \eqref{Moreau identity} can be rewritten into
	\begin{equation}\label{Y to prove}
	\bm{v} = \mathcal{S}\bm{x} +(\mathcal{S}^*)^{-1} \bm{y}.
	\end{equation}
	Our goal can be stated as proving the following relation:
	\begin{equation}\label{to_pro}
	\bm{y} = \mathcal{S}^*(\bm{v} - \mathcal{S}\bm{x}). \tag{goal 0}
	\end{equation}
	Our goal is hence to prove \eqref{to_pro}.
	
	To this end, invoke the first definition in \eqref{defs}. Since $ \bm{x} $ is a minimizer, by the first-order optimality, we obtain
	\begin{align}
	\mathcal{S}^*(\bm{v} - \mathcal{S}\bm{x}) \in \partial {f}(\bm{x}) 
	\iff	& \bm{x} \in \partial f^*(  \mathcal{S}^*\bm{v} - \mathcal{S}^*\mathcal{S}\bm{x} ) \nonumber\\
	\iff	& \mathcal{S}^{-1}(\bm{v} -  \bm{v} + \mathcal{S}\bm{x})  \in \partial f^*(  \mathcal{S}^*\bm{v} - \mathcal{S}^*\mathcal{S}\bm{x} ) \nonumber\\
	\iff	& \mathcal{S}^{-1}(\bm{v} - (\mathcal{S}^*)^{-1}\mathcal{S}^*(  \bm{v} - \mathcal{S}^*\bm{x})  ) \in \partial f^*(  \mathcal{S}^*\bm{v} - \mathcal{S}^*\mathcal{S}\bm{x} ) \nonumber\\
	\iff	& 	\mathcal{S}^{-1}(\bm{v} - 	(\mathcal{S}^*)^{-1}\bm{t}) \in \partial f^*(\bm{t}) \nonumber\\
	\iff	& 	\bm{t} = (\mathcal{M}^{-1} + \partial f^* )^{-1} \mathcal{S}^{-1}\bm{v},
	\label{v-sx}
	\end{align}
	with equality holds due to the above resolvent mapping being single-valued, see \cite{bauschke2017convex} (alternatively, see later Proposition \ref{prop_firm2} that the proximal operator is single-valued),
	where variable substitution $ \bm{t} \eqdef \mathcal{S}^*(  \bm{v} - \mathcal{S}^*\bm{x})  $ is used.
	
	Now, invoke the second definition in \eqref{defs}. Since $ \bm{y} $ is a minimizer, by the first-order optimality, we obtain
	\begin{align}\label{f^*_Y}
			&	\mathcal{S}^{-1}(\bm{v} - 	(\mathcal{S}^*)^{-1}\bm{y}) \in \partial f^*(\bm{y})\nonumber\\
	\iff	&  		\bm{y} = (\mathcal{M}^{-1} + \partial f^* )^{-1} \mathcal{S}^{-1}\bm{v}
	\end{align}
	Comparing \eqref{v-sx} and \eqref{f^*_Y}, we obtain
	\begin{equation}
	\bm{y} = \bm{t} = \mathcal{S}^*(\bm{v} - \mathcal{S}\bm{x}).
	\end{equation}
	That is, we have shown \eqref{to_pro}. The proof is therefore concluded.
\end{proof}

\subsubsection{Translation rule}
The classical proximal operator  and the new one can be  easily converted. 
For  the sake of generality, we present  the  metric environment case, which subsumes  the scalar one.
	\begin{lem}[translation rule]\label{lem_rel2}
	The classical proximal operator  defined in \eqref{cla_prox}, and  the  equilibrate proximal operator defined in \eqref{def_new}  can be  converted through
	\begin{align}
	&(\text{Remove extra-scaling}) \quad		\textbf{Prox}_{f\mathcal{S}^{-1}}( \bm{v} )  = \quad\mathcal{S}\,\, \circ \,	\textbf{Prox}_{ f}^{\mathcal{M}^{-1}}\circ\mathcal{S}^{-1} \,(\bm{v}),  \nonumber\\
	&\,\,(\text{Equip  extra-scaling})\quad\,\,\,\,  	\textbf{Prox}_{ f}^{\mathcal{M}^{-1}}(\bm{v}) = \,\mathcal{S}^{-1}\circ\textbf{Prox}_{f\mathcal{S}^{-1}}\circ\mathcal{S} \,\,\, ( \bm{v} ).
	\end{align}
\end{lem}
\begin{proof}
	For the first relation, from the right-hand side, we have 
	\begin{equation}
	\mathcal{S}\, 	\textbf{Prox}_{ f}^{\mathcal{M}^{-1}}(\mathcal{S}^{-1}\bm{v})=\mathcal{S}\,\underset{\bm{z}}{\text{argmin}} \,\,  f(\bm{z}) + \frac{1}{2}\Vert \bm{z} - \mathcal{S}^{-1}\bm{v}\Vert^2_\mathcal{M}
	=\mathcal{S}\,\underset{\bm{z}}{\text{argmin}} \,\,  f(\bm{z}) + \frac{1}{2}\Vert \mathcal{S}\bm{z} - \bm{v}\Vert^2
	= \textbf{Prox}_{f\mathcal{S}^{-1} }(\bm{v})
	\end{equation}
	For the second relation, from the right-hand side, we have 
	\begin{equation}
	\mathcal{S}^{-1}\textbf{Prox}_{f\mathcal{S}^{-1}}( \mathcal{S}\bm{v} ) 
	=\mathcal{S}^{-1}\mathcal{S}\,\underset{\bm{z}}{\text{argmin}} \,\,f(\bm{z}) + \frac{1}{2}\Vert \mathcal{S}\bm{z} - \mathcal{S}\bm{v}\Vert^2
	= \underset{\bm{z}}{\text{argmin}} \,\,f(\bm{z}) + \frac{1}{2}\Vert \bm{z} - \bm{v}\Vert^2_\mathcal{M}
	=	\textbf{Prox}_{ f}^{\mathcal{M}^{-1}} (\bm{v}) 
	\end{equation}
	 The proof is now concluded.
\end{proof}

\subsubsection{Convergence property}
For the sake of concreteness, we prove that the  equilibrate proximal operator  is firmly non-expansive and single-valued. This is the key property to ensure   convergence and is used in the next section.
\begin{prop}\label{prop_firm2}
	$ 	\textbf{Prox}_{f \rho} $   and its generalization 
	$\textbf{Prox}_{f\mathcal{S}}$, defined in \eqref{new_def1} and \eqref{def_new}, respectively, are firmly nonexpansive and single-valued.	
\end{prop}
\begin{proof}
	To start, we need some lemmas.
	\begin{lem}\cite[Proposition 6.19]{bauschke2017convex}\label{lem1}
		Let $ C$ be a convex subset of $ \mathscr{H} $, let $ \mathscr{L} $ be a real Hilbert space, and let  $D$ be a convex subset of $ \mathscr{L} $. Suppose $ D \bigcap \text{int}\,\, \mathcal{L}(C) \neq \emptyset$ or $ \mathcal{L}(C) \bigcap \text{int}\,\, D \neq \emptyset$. Then, $ 0\in \text{sri}\, (D - \mathcal{L}(C) ) $.
	\end{lem}
	
	\vspace{1pt}
	
	\begin{lem}\cite[Corollary 16.53]{bauschke2017convex}\label{dom_ran}
		Let $ f $ be a CCP function and $  \mathcal{L} \in \mathfrak{B}(\mathscr{H}, \mathscr{L}) $. Suppose $ 0\in \text{sri}\, (\text{dom}(f) -\text{ran}(\mathcal{L}) ) $. Then, 
		\begin{equation*}
		\partial (f\circ\mathcal{L}) = \mathcal{L}^*\circ\partial f\circ\mathcal{L}.
		\end{equation*}
	\end{lem}

	\begin{lem}\cite[Proposition 23.25]{bauschke2017convex}\label{coro_maxm}
		Let $ \mathscr{L} $ be a real Hilbert space,  suppose  $ \mathcal{L} \in \mathfrak{B}(\mathscr{H}, \mathscr{L}) $ is such that $ \mathcal{L}\mathcal{L}^* $ is invertible, let $\mathcal{A}: \mathscr{L} \rightarrow2^\mathscr{L}$ be maximal monotone, and set $ \mathcal{B} = \mathcal{L}^*\mathcal{A}\mathcal{L} $. Then, 	
		$\mathcal{B}: \mathscr{H}\rightarrow 2^\mathscr{H}$ is maximal monotone.	
	\end{lem}
	
	Below, we consider $\textbf{Prox}_{f\mathcal{S}}$, which includes the scalar parameter case. In view of Lemma \ref{lem1}, due to  $ \text{ran}(\mathcal{S}) \bigcap \text{int}\,\, \text{dom}(f) \neq \emptyset$, we have  $ 0\in \text{sri}\, (\text{dom}(f) -\text{ran}(\mathcal{L}) ) $. Then, by Lemma \ref{dom_ran}, we have
	$ \partial (f\circ\mathcal{S}) =  \mathcal{S}^*\circ\partial f\circ\mathcal{S} $. 
	By Lemma \ref{coro_maxm}, the operator $ \partial (f\circ\mathcal{S}) $ is maximal monotone.	
	
	Now, invoke the definition of $ \textbf{Prox}_{f\mathcal{S}} $, which yields
	\begin{align}
	\textbf{Prox}_{f\mathcal{S}} (\bm{v}) = \underset{\bm{z}}{\text{argmin}} \,\, f(\mathcal{S}\bm{z})+ \frac{1}{2}\Vert \bm{z} - \bm{v}\Vert^2 
	\Longleftrightarrow& 	\,\bm{v} - \bm{z}  \in  \partial (f\circ\mathcal{S}) \,\bm{z},   \\
	\Longleftrightarrow&    \,\bm{z}   \in (\mathcal{I} + \partial (f\circ\mathcal{S}))^{-1}\,\bm{v}.\nonumber
	\end{align}
	Following from above, operator $ \partial (f\mathcal{S}) $ is maximal monotone. Then, the last line is in fact single-valued. Furthermore, due to the maximal monotone property, $ 	\textbf{Prox}_{f\mathcal{S}}  =  (\mathcal{I} + \partial (f\circ\mathcal{S}))^{-1} $ is firmly nonexpansive owing to \cite{bauschke2017convex}. The proof is therefore concluded.
\end{proof}

\subsection{Equilibrate monotone operators}

Here, we consider the  \textit{equilibrate parametrization} in the monotone operator case. It is achieved by associating the parameter to  both the domain  and the range of a monotone operator, rather than conventionally the range  only.
Below,  we present some brief arguments by appealing to the equivalence  between the proximal operator and the resolvent.

Recall from \eqref{equ_prox} that the resolvent and the proximal operator are equivalent given $ f \in \Gamma_0  (\mathscr H)$. In this case,
\begin{equation}
	\textbf{Prox}_{f \rho}  = 
J_{\rho\,\mathcal{A}\,\rho } = \bigg(  \mathcal{I} +\rho\,\mathcal{A}\,\rho \bigg)^{-1}, \qquad \rho \neq 0,
\end{equation}
with  $  \mathcal{A} =  \partial f$ and $  f \in \Gamma_0 (\mathscr H)$,
where,   for   light of notations, we denote $ \rho\,\mathcal{A}\,\rho  = (\rho\mathcal{I}) \circ \mathcal{A}\circ (\rho\mathcal{I})  $.

Suppose the operator $ \mathcal{A}: \mathscr H \rightarrow 2^\mathscr H$ is maximal monotone. Then, $ \rho\,\mathcal{A}\,\rho $ is also maximal monotone.
The identity relation \eqref{re_id} is extended into
\begin{equation}
J_{\rho\,\mathcal{A}\,\rho} + J_{ \rho^{-1} \mathcal{A}^{-1} \rho^{-1} } = \mathcal{I}.
\end{equation}

Similarly, in a metric space environment  $  \mathscr {H}_{\mathcal{M}}$ with decomposed metric $ \mathcal{S} $ as in Lemma \ref{lem_emb}, we have
\begin{equation}
	\textbf{Prox}_{f \mathcal{S}} = 
J_{\mathcal{S}^* \mathcal{A}\mathcal{S}} = \big(   \mathcal{I} + \mathcal{S}^* \circ\mathcal{A}\circ \mathcal{S} \big)^{-1} ,
\end{equation}
with  $  \mathcal{A} =  \partial f$,  $  f \in \Gamma_0 (\mathscr H)$ and $ \mathcal{S} \in \mathfrak B (\mathscr H) $ being bijective.

Suppose the operator $ \mathcal{A}: \mathscr H \rightarrow 2^\mathscr H$ is maximal monotone. Then, $\mathcal{S}^* \mathcal{A}\mathcal{S} $ is also maximal monotone.
The identity relation \eqref{re_id} is extended into
\begin{equation}
J_{\mathcal{S}^* \mathcal{A}\mathcal{S}} + J_{ (\mathcal{S}^*)^{-1} \mathcal{A}^{-1}\mathcal{S}^{-1}} = \mathcal{I}.
\end{equation}

\subsection{Extension:  injective parameters}\label{sec_inj}
In the previous sections, we consider parameter $ \mathcal{S} \in \mathfrak B (\mathscr H) $ to be bijective.
In fact,  under one additional condition, this requirement can be relaxed into being injective. This aspect will be useful in the next section,  enabling a unified treatment of the metric parameter and injective matrices (in the ADMM constraint). 

The following equilibrate proximal operators are also well-defined:
\begin{prop}
	Given $ f \in \Gamma_0 (\mathscr H) $ and $ \mathcal D  \in \mathfrak B (\mathscr H,  \mathscr L) $ being injective,  the following \textit{equilibrate proximal operators} are firmly non-expansive and single-valued:
	\begin{equation}
	\textbf{Prox}_{f\mathcal{D}^{-1}} (\bm{v}) = \underset{\bm{z}_1 \in \mathscr L }{\text{argmin}} \,\, f\left(\mathcal{D}^{-1}\bm{z}_1\right)+ \frac{1}{2}\Vert \bm{z}_1 - \bm{v}\Vert^2
	=  \mathcal{D} \, \underset{\bm{z}_2 \in \mathscr H}{\text{argmin}} \,\, f(\bm{z}_2)+ \frac{1}{2}\Vert \mathcal{D}\bm{z}_2 - \bm{v}\Vert^2,
	\end{equation}
		\begin{equation}
	\textbf{Prox}_{f\mathcal{D}^{*}} (\bm{v}) = \underset{\bm{z}_1 \in \mathscr L}{\text{argmin}} \,\, f\left(\mathcal{D}^*\bm{z}_1\right)+ \frac{1}{2}\Vert \bm{z}_1 - \bm{v}\Vert^2
	=  ( \mathcal{D}^*)^{-1}\, \underset{\bm{z}_2 \in \mathscr H}{\text{argmin}} \,\, f(\bm{z}_2)+ \frac{1}{2}\Vert( \mathcal{D}^*)^{-1}\bm{z}_2 - \bm{v}\Vert^2.
	\end{equation}
\end{prop}

\begin{proof}
The proof follows  the same  arguments of  the  bijective case in Proposition \ref{prop_firm2}, with one additional step. That is, in view of Lemma \ref{coro_maxm}, we need $ \mathcal{L}\mathcal{L}^* $ being invertible to guarantee the maximal monotone property, which hold for the following two cases:  let $ \mathcal{L}_1 =  \mathcal{D}^{-1}$ and $ \mathcal{L}_2 =  \mathcal{D}^*$, then
\begin{equation}
 \mathcal{L}_1\mathcal{L}^*_1 = \mathcal{D}^{-1}  ({\mathcal{D}}^*)^{-1} = (\mathcal{D}^*\mathcal{D})^{-1}, 
 \qquad 
 \mathcal{L}_2\mathcal{L}^*_2 = \mathcal{D}^*\mathcal{D},
\end{equation}
which are clearly invertible for an injective operator $ \mathcal{D} $. The proof is now concluded.
\end{proof}



\section{Equilibrate ADMM}\label{sec_admm}
In the  previous section, we propose the \textit{equilibrate parametrization} that naturally avoids  extra scaling. Here, we would  exploit its potential benefits by applying it to the ADMM algorithm.
We would obtain a direct fixed-point analysis  framework for  the general ADMM, previously not available (the conventional way is by appealing to the DRS, see \cite[sec. 3.1, 3.2]{ryu2022large},  \cite{poon2019trajectory}).
Moreover, we prove that the self-duality will be restored,  implying  direct access to an unscaled fixed-point.

\subsection{Metric space iterates}

Consider the following general problem:
\begin{align}\label{ab_prob}
\underset{\bm{x},\bm{z}}{\text{minimize}}
\quad&   f(\bm{x}) + g(\bm{z})\quad \nonumber\\
\text{subject\,to}
\quad&  \bm{A}\bm{x} - \bm{B}\bm{z} = \bm{c} ,
\end{align}
with $ f, g \in \Gamma_0 (\mathscr H)  $ and matrices $ \bm{A}, \bm{B} $ have full columnn-rank (injective).

Given a metric space environment $ \mathscr{H}_\mathcal{M} $ with decomposed  metric $ \mathcal{S}  $  as in Lemma \ref{lem_emb},
the augmented Lagrangian  can be written into
\begin{align}\label{gen_L}
\mathcal{L}_\mathcal{M}(\bm{x},\bm{z},\bm{\lambda}) 
\,=\,	&  f(\bm{x}) + g (\bm{z})  + \langle   \bm{\lambda}, \, \bm{A}\bm{x} - \bm{B}\bm{z} - \bm{c}  \rangle + \frac{1}{2}\Vert  \bm{A}\bm{x} - \bm{B}\bm{z} - \bm{c}  \Vert^2_\mathcal{M} \nonumber\\
\,=\, 		&  f(\bm{x}) + g (\bm{z})  + \langle   \bm{\lambda}, \, \bm{A}\bm{x} - \bm{B}\bm{z} - \bm{c}\rangle + \frac{1}{2}\Vert \mathcal{S} \left(\bm{A}\bm{x} - \bm{B}\bm{z} - \bm{c}\right)  \Vert^2.
\end{align}
The  ADMM iterates are then given by
\begin{align}
\bm{x}^{k+1} =\,\,& \underset{\bm{x}}{\text{argmin}} \,\, 	\mathcal{L}_\mathcal{M}(\bm{x},\bm{z}^k,\bm{\lambda}^k)   
			 \,\quad=\,\, \underset{\bm{x}}{\text{argmin}} \,\, f(\bm{x}) + \frac{1}{2}\Vert \mathcal{S} \left(\bm{A}\bm{x} - \bm{B}\bm{z}^k - \bm{c}\right) +  (\mathcal{S}^*)^{-1} \bm{\lambda}^k \Vert^2  \nonumber\\				
\bm{z}^{k+1} =\,\, & \underset{\bm{z}}{\text{argmin}} \,\, 	\mathcal{L}_\mathcal{M} (\bm{x}^{k+1},\bm{z},\bm{\lambda}^k)      
			 \,=\,\, \underset{\bm{z}}{\text{argmin}} \,\,  g (\bm{z})  + \frac{1}{2}\Vert \mathcal{S} \left(\bm{A}\bm{x}^{k+1} - \bm{B}\bm{z} - \bm{c}\right) +  (\mathcal{S}^*)^{-1} \bm{\lambda}^k \Vert^2 \nonumber\\
\bm{\lambda}^{k+1} =\,\,  &\bm{\lambda}^{k} + \mathcal{M} \left( \bm{A}\bm{x}^{k+1} - \bm{B}\bm{z}^{k+1}  - \bm{c} \right). 	
\end{align}	
which can be written in terms of the \textit{equilibrate proximal operator}, as 
\begin{align}\label{R-ADMM}
\bm{x}^{k+1} =\,\,& (\mathcal S\bm{A})^{-1} \circ \textbf{Prox}_{f(\mathcal S\bm{A})^{-1}} \left(\quad\mathcal S   \bm{c} + \mathcal S\bm{B}\bm{z}^k  \quad - (\mathcal{S}^*)^{-1} \bm{\lambda}^k\right)  \nonumber\\				
\bm{z}^{k+1} =\,\, & (\mathcal S\bm{B})^{-1} \circ \textbf{Prox}_{g(\mathcal S\bm{B})^{-1}} \left(-\,\mathcal S   \bm{c} + \mathcal S\bm{A}\bm{x}^{k+1}  + (\mathcal{S}^*)^{-1} \bm{\lambda}^k\right) \nonumber\\
\bm{\lambda}^{k+1} =\,\,  &\bm{\lambda}^{k} + \mathcal S^*\mathcal S\left( \bm{A}\bm{x}^{k+1} - \bm{B}\bm{z}^{k+1} - \bm{c}\right).	\tag{E-ADMM}
\end{align}	
The above iterates are well-defined for injective mappings $ \bm{A} $ and $ \bm{B} $,
owing to Section \ref{sec_inj}.  We refer to the above as the \textit{Equilibrate ADMM} (abbr. E-ADMM).

To simplify the iterates expression, we  define the  following variable substitutions:
\begin{equation}\label{sub}
\widetilde{\bm{x}}^{k} \eqdef \mathcal S\bm{A}\bm{x}^{k}, \qquad \widetilde{\bm{z}}^{k} \eqdef \mathcal S\bm{B}\bm{z}^{k}, \qquad \widetilde{\bm{\lambda}}^{k} \eqdef (\mathcal{S}^*)^{-1}  \bm{\lambda}^{k}
\qquad \widetilde{\bm{c}} \eqdef\mathcal S \bm{c}.
\end{equation}
We refer to the following as the \textit{equilibrate scaled form}:
\begin{align}\label{scaled_form}
\widetilde{\bm{x}}^{k+1} =\,\,& \textbf{Prox}_{f(\mathcal S\bm{A})^{-1}} \left( \quad\widetilde{\bm{c}} + \widetilde{\bm{z}}^k  \quad- \widetilde{\bm{\lambda}}^{k}\right)  \nonumber\\				
\widetilde{\bm{z}}^{k+1} =\,\, &  \textbf{Prox}_{g(\mathcal S\bm{B})^{-1}} \left(-\,\widetilde{\bm{c}}+ \widetilde{\bm{x}}^{k+1} + \widetilde{\bm{\lambda}}^{k}\right) \nonumber\\
\widetilde{\bm{\lambda}}^{k+1} =\,\,  &\widetilde{\bm{\lambda}}^{k} + \widetilde{\bm{x}}^{k+1}- \widetilde{\bm{z}}^{k+1} - \widetilde{\bm{c}}. 	\tag{scaled form}
\end{align}


\subsection{Iterates convergence}
Here,  without additional assumptions, we prove the convergence of the  ADMM iterates \eqref{R-ADMM}  in  $ \mathscr{H}_\mathcal{M} $.
To start, we need one lemma. 
\begin{lem}\label{post_comp}
Let $ h(\bm{z}) \eqdef f(\bm{z}) - \langle \bm{c}, \bm{z} \rangle$, then
\begin{equation}
\textbf{Prox}_{h\mathcal{S}^{-1}} (\bm{v}) = \textbf{Prox}_{{f}\mathcal{S}^{-1}}  \big(\,  \bm{c} + \bm{v} \,\big)
\end{equation}

\begin{proof}
By definition  \eqref{def_new}, we have
	\begin{align}
\textbf{Prox}_{h\mathcal{S}^{-1}} (\bm{v}) 
=&  \underset{\bm{z}_1}{\text{argmin}} \,\, h\circ\mathcal{S}^{-1}(\bm{z}_1)+ \frac{1}{2}\Vert\bm{z}_1 - \bm{v}\Vert^2 \nonumber\\
=&  \underset{\bm{z}_1}{\text{argmin}} \,\,  f\circ\mathcal{S}^{-1} (\bm{z}_1) - \langle \bm{c}, \bm{z}_1 \rangle+ \frac{1}{2}\Vert\bm{z}_1 - \bm{v}\Vert^2 \nonumber\\
=&  \underset{\bm{z}_1}{\text{argmin}} \,\,  f\circ\mathcal{S}^{-1} (\bm{z}_1) + \frac{1}{2}\Vert\bm{z}_1 -\bm{c} - \bm{v}\Vert^2 \nonumber\\
=&  \textbf{Prox}_{{f}\mathcal{S}^{-1}} (  \bm{c} + \bm{v}  ) ,
\end{align}
which concludes the proof.
\end{proof}
\end{lem}


Now, we are ready to provide a fixed-point characterization and establish convergence.
\begin{prop}\label{fpi}
\eqref{R-ADMM} admits the following fixed-point characterization:
\begin{equation}\label{fpi_admm}
\bm{\zeta}^{k+1} = \mathcal{F}_\text{ADMM} \,\, \bm{\zeta}^{k},
\end{equation}
with
\begin{equation}\label{fp}
\bm{\zeta}^{k+1} 
=\,\,   \mathcal S\bm{A}\bm{x}^{k+1} +  (\mathcal{S}^*)^{-1} \bm{\lambda}^{k}, 
\qquad \mathcal{F}_{ADMM} =  \frac{1}{2}(2\textbf{Prox}_{\bar{f}(\mathcal S\bm{A})^{-1}} - \mathcal{I})\circ(2\textbf{Prox}_{ \widehat{g}(\mathcal S\bm{B})^{-1}} - \mathcal{I}) + \frac{1}{2}\mathcal I ,
\end{equation}
where 
\begin{equation}\label{pert}
\bar{f}(\bm{x}) \eqdef f(\bm{x}) - \langle \mathcal S\bm{c}, \bm{x} \rangle , 
\qquad \widehat{g}(\bm{z}) \eqdef g(\bm{z}) + \langle \mathcal S\bm{c}, \bm{z} \rangle. 
\end{equation}
The sequence $ \{\bm{\zeta}^{k}\} $ converges to a fixed-point $ \bm{\zeta}^\star \in \text{Fix}\,\, \mathcal{F}_\text{ADMM}  $ (if it exists).
\end{prop}

\begin{proof}
	For light of notation, we use \eqref{scaled_form}, with 
	\begin{equation}
	\widetilde{\bm{x}}^{k} \eqdef \mathcal S\bm{A}\bm{x}^{k}, \qquad \widetilde{\bm{z}}^{k} \eqdef \mathcal S\bm{B}\bm{z}^{k}, \qquad \widetilde{\bm{\lambda}}^{k} \eqdef (\mathcal{S}^*)^{-1}  \bm{\lambda}^{k}
	\qquad \widetilde{\bm{c}} \eqdef\mathcal S \bm{c}.
	\end{equation}
	In view of the  $ \bm{\lambda} $-update, we have
	\begin{align}
	\widetilde{\bm{\lambda}}^{k}
	 =\,\,  &\widetilde{\bm{\lambda}}^{k-1} + \widetilde{\bm{x}}^{k} - \widetilde{\bm{z}}^{k} - \widetilde{\bm{c}} \nonumber\\
	 =\,\,  & \bm{\zeta}^{k}  - \widetilde{\bm{z}}^{k} - \widetilde{\bm{c}}. 
	\end{align}
In view of \eqref{scaled_form}, we obtain
	\begin{align}\label{pri_fpi}
\bm{\zeta}^{k+1} 
=\,\,  &  \widetilde{\bm{x}}^{k+1} + \widetilde{\bm{\lambda}}^{k} \nonumber\\
=\,\,  & \textbf{Prox}_{f  (\mathcal S\bm{A})^{-1}} (2\widetilde{\bm{c}} +  2 \widetilde{\bm{z}}^{k}  - \bm{\zeta}^{k}) + \bm{\zeta}^{k+1}  - \widetilde{\bm{z}}^{k+1} - \widetilde{\bm{c}} , \nonumber\\
=\,\,  & \textbf{Prox}_{ f (\mathcal S\bm{A})^{-1}} \bigg(\widetilde{\bm{c}} + 2\textbf{Prox}_{ g(\mathcal S\bm{B})^{-1}}(\bm{\zeta}^{k} - \widetilde{\bm{c}} )- (\bm{\zeta}^{k} - \widetilde{\bm{c}} ) \bigg)   + (\bm{\zeta}^{k} - \widetilde{\bm{c}}) - \textbf{Prox}_{ g(\mathcal S\bm{B})^{-1}}(\bm{\zeta}^{k} - \widetilde{\bm{c}} ), \nonumber\\
=\,\,  & \textbf{Prox}_{  {f}(\mathcal S\bm{A})^{-1} }\circ \bigg(\,\widetilde{\bm{c}} + \big(2\textbf{Prox}_{ {g}(\mathcal S\bm{B})^{-1}} - \mathcal{I}\big)   (\bm{\zeta}^{k} - \widetilde{\bm{c}} )\bigg) - \frac{1}{2} \bigg( \,\widetilde{\bm{c}} +
\big( 2\textbf{Prox}_{ {g}(\mathcal S\bm{B})^{-1}} - \mathcal{I}\big)  (\bm{\zeta}^{k} - \widetilde{\bm{c}} ) \bigg)  + \frac{1}{2}\bm{\zeta}^{k}, \nonumber\\
=\,\,  & \left( \frac{1}{2}(2\textbf{Prox}_{\bar{f}(\mathcal S\bm{A})^{-1}} - \mathcal{I})\circ(2\textbf{Prox}_{ \widehat{g}(\mathcal S\bm{B})^{-1}} - \mathcal{I}) + \frac{1}{2}\mathcal I   \right) \,   \bm{\zeta}^{k},
\end{align}
where the last line  uses Lemma \ref{post_comp}.
Invoking the substitution definitions gives  the  original fixed-point expression:
\begin{equation}\label{ADMM_fix}
\bm{\zeta}^{k+1} 
=\,\,   \mathcal S\bm{A}\bm{x}^{k+1} +  (\mathcal{S}^*)^{-1} \bm{\lambda}^{k}.
\end{equation} 
At last, the operator $ \mathcal{F}_{ADMM}$ is 1/2-averaged, and therefore firmly non-expansive, recall \eqref{rel00}. It follows that the iterates converge to a fixed-point $ \bm{\zeta}^\star \in Fix\,\, \mathcal{F}_{ADMM} $ (if it exists). The proof is now concluded.
\end{proof}

\subsection{Self-duality}\label{admm_self}
Recall  from Sec. \ref{self_duality} that the self-duality  does not hold for the classical parametrization.  Here, using  the  \textit{equilibrate parametrization},  the self-duality naturally holds.

To show this,
recall the equilibrate Moreau identity as in \eqref{Moreau identity}:
\begin{equation}
\bm{v}  =\,\textbf{Prox}_{f\mathcal{S}^{-1}} (\bm{v}) +  \textbf{Prox}_{f^*{\mathcal{S}^*}} (\bm{v}),
\end{equation}
with $ f  \in \Gamma_0 (\mathscr H) $, and  $ \mathcal{S} \in \mathfrak B (\mathscr H) $ being bijective.
It follows that
\begin{align}
&  \quad \big(\mathcal I - \textbf{Prox}_{f\mathcal{S}^{-1}}\big) \,\,\bm{v} =   \quad\textbf{Prox}_{f^*{\mathcal{S}^*}} \,\,(\bm{v})\nonumber\\
\iff
&  \quad \big(\mathcal I - \textbf{Prox}_{f\mathcal{S}^{-1}}\big) \,\,\bm{v} =   -\,\,\textbf{Prox}_{f^*\mathcal{S}^*(-\mathcal{I})} \,\,(-\bm{v})\nonumber\\
\iff
&\,\,\,   \big(2\textbf{Prox}_{f\mathcal{S}^{-1}} - \mathcal I\big) \,\,\bm{v} =   \quad\big(2\textbf{Prox}_{f^*{\mathcal{S}^*}(-\mathcal{I})} - \mathcal I \big)   \,\,(-\bm{v}).
\end{align}
Similarly, given $ g  \in \Gamma_0 (\mathscr H) $,
\begin{align}
 &  \quad\,\, \big(\mathcal I - \textbf{Prox}_{g\mathcal{S}^{-1}}\big) \,\,\,\bm{v} =  \,\,\textbf{Prox}_{g^*{\mathcal{S}^*}} \,\,(\bm{v})\nonumber\\
 \iff
 &  \quad\, \big(\mathcal I - 2\textbf{Prox}_{g\mathcal{S}^{-1}}\big) \,\,\bm{v} =   \big(\textbf{Prox}_{g^*{\mathcal{S}^*}}  - \textbf{Prox}_{g\mathcal{S}^{-1}}\big) \,\,(\bm{v})\nonumber\\
 \iff
 & - \big(2\textbf{Prox}_{g\mathcal{S}^{-1}} - \mathcal I\big) \,\,\bm{v} =   \big(2\textbf{Prox}_{g^*\mathcal{S}^*} - \mathcal I \big)   \,\,\bm{v}.
\end{align}
Using the above two relations,  we obtain
\begin{equation}
\underbrace{(2\textbf{Prox}_{\bar{f}(\mathcal S\bm{A})^{-1}} - \mathcal{I})\circ(2\textbf{Prox}_{ \widehat{g}(\mathcal S\bm{B})^{-1}} - \mathcal{I})}_{\mathcal{F}_\text{pri}} 
= \underbrace{(2\textbf{Prox}_{\bar{f}^*\bm{A}^T\mathcal S^*(-\mathcal{I})} - \mathcal{I})\circ(2\textbf{Prox}_{ \widehat{g}^*\bm{B}^T\mathcal S^*} - \mathcal{I})}_{\mathcal{F}_\text{dual}} .
\end{equation}
Recall  from Sec. \ref{self_d}, the above implies that the self-duality holds. Therefore, the primal  fixed-point (see Proposition \ref{fpi}) coincides with the  dual one, being
the unscaled fixed-point:
\begin{equation}\label{un_fp}
\bm{\zeta}^\star
=\,\,   \mathcal S\bm{A}\bm{x}^\star +  (\mathcal{S}^*)^{-1} \bm{\lambda}^\star.
\end{equation}


\section{Equivalence: preconditioning \& metric selection}\label{sec_frame}
In the previous section, we  apply \textit{equilibrate  parametrization}  to the ADMM algorithmic framework. 
Here, we apply it to the optimization problem  itself.  We will obtain some intrinsic structure insights.
Quite remarkably, the  preconditioning techniques  and the step-size/metric parameter selection  are  shown to be  equivalent. Specifically,

\vspace{6pt}
%

$\bullet$\, Constraint preconditioning  is equivalent to selecting  the variable metric; 

\vspace{6pt}
$\bullet$\, Scaling the objective function is equivalent to  selecting the scalar step-size. 

\vspace{6pt}

Due to the scalar step-size can be absorbed into the variable metric,  the above essentially reduces to a single issue, 
the metric selection.





\subsection{A unified view}
We will show that, owing to the \textit{equilibrate  parametrization}, the (ADMM type) primal and dual problems can be rewritten into a completely symmetric form.

First, we  recall the classical characterizations.
Consider the following primal problem:
\begin{align}\label{pri1}
\underset{\bm x,\bm z}{\text{minimize}}\quad&   f(\bm x) + g(\bm z)\quad  \nonumber\\
\text{subject\,to}\quad&   \bm{A}\bm x - \bm{B}\bm z =  \bm c. \tag{standard}
\end{align} 
with the Lagrange function
\begin{align}\label{or_L}
\mathcal{L}(\bm{x},\bm{z},\bm{\lambda}) 
\,=\,	&  f(\bm{x}) + g (\bm{z})  + \langle   \bm{\lambda}, \, \bm{A}\bm{x} - \bm{B}\bm{z} - \bm{c}  \rangle.
\end{align}
The Fenchel dual problem  is well-known to be
\begin{equation}\label{dual1}
\underset{\bm\lambda}{\text{minimize}}\quad     f ^* (-\bm{A}^T\bm\lambda)	+   g^*(\bm{B}^T\bm\lambda) 	+ \langle\bm\lambda,\, \bm c \rangle.
\end{equation}
From above, 
we see that \eqref{pri1}  contains two variables and one  constraint, while the dual \eqref{dual1} admits only one variable and is unconstrained. Their connections are not  clear under the current characterizations.
We are therefore motivated  to propose a new one, which emphasizes the  underlying symmetry.
\begin{prop}[unified primal-dual problems]\label{prop_uni_pd}
	The standard  primal problem \eqref{pri1} admits the following  unified    characterizations:
	\begin{align}
	&\underset{\widetilde{\bm x}}{\text{minimize}}\qquad\quad\qquad\widetilde{f} \quad \big(\widetilde{\bm x}\big) \,\,\, + \,\, \widetilde{g}\,\,\, \big( \widetilde{\bm x} - \widetilde{\bm c} \big) ,   \tag{Primal}\\
	&\underset{\widetilde{\bm\lambda}}{\text{minimize}}\quad    \bigg(\widetilde{f}\circ(-\mathcal{I})\bigg)^*\big(\widetilde{\bm\lambda}\big)\,\,	+  \,\,  \widetilde{g}^*\, \big(\widetilde{\bm\lambda} - \widetilde{\bm c}\big) ,  \tag{Dual}
	\end{align}
For the \textit{equilibrate parametrization}, with $ \mathcal S  \in \mathfrak B (\mathscr H) $ being bijective, the  following variable substitutions are used:
\begin{align}\label{equ_para}
\,\,&(\text{Function}) \qquad\qquad	\widetilde{f} =   f \circ (\mathcal{S}\bm{A})^{-1}, \qquad\quad  \widetilde{g} =   g \circ (\mathcal{S}\bm{B})^{-1},\nonumber\\
\,\,&(\text{Variable} ) \qquad\qquad \,\,	\widetilde{\bm x}=  \mathcal{S}\bm{A}\, \bm  x, \qquad\qquad\quad\,\,    \widetilde{\bm\lambda} = (\mathcal{S}^*)^{-1}\,\bm\lambda,  \qquad\quad \widetilde{\bm c} = \mathcal{S}\bm c.
\end{align}
For the classical  left-parametrization,  with scalar $\gamma>0$, the  following variable substitutions are used:
\begin{align}\label{class_va}
&(\text{Function})\qquad \qquad	\widetilde{f} =   \frac{1}{\gamma}\, f \circ \bm{A}^{-1}, \qquad\quad\,\,\,  \widetilde{g} =  \frac{1}{\gamma}\,  g \circ \bm{B}^{-1},\nonumber\\
&(\text{Variable} )\qquad\qquad  \,\,	\widetilde{\bm x}=  \bm{A}\, \bm  x, \qquad\qquad\qquad    \widetilde{\bm\lambda} = \frac{1}{\gamma}\,\bm\lambda,   \qquad\qquad\quad \widetilde{\bm c} = \bm c.
\end{align}

Moreover, \eqref{pri1} admit the following monotone inclusion characterizations:
\begin{align}
&\text{find} \,\, \widetilde{\bm x}\in \mathscr H \quad \text{such that}\:\: 	0 \in \,\,\mathcal{A}\quad(\widetilde{\bm x} )    +\,\, \mathcal{B}\,\,\,\, (\widetilde{\bm x} - \widetilde{\bm c} ), \nonumber\\
&\text{find} \,\, \widetilde{\bm\lambda}\in \mathscr H \quad \text{such that}\:\:   		0 \in {\mathcal{A'}}^{-1} (\widetilde{\bm\lambda} ) + {\mathcal{B}}^{-1} (\widetilde{\bm\lambda}- \widetilde{\bm c} ),  
\end{align}
where
\begin{equation}
\mathcal{A} = \partial \widetilde{f},\qquad   \mathcal{B}  = \partial \widetilde{g},
\end{equation}
with functions $\widetilde{f}, \widetilde{g} $ and variables $ \widetilde{\bm x}, \widetilde{\bm\lambda}, \widetilde{\bm c} $ defined in \eqref{equ_para} for \textit{equilibrate parametrization},  and \eqref{class_va} for classical parametrization.
\end{prop}
\begin{proof}
\eqref{equ_para} is shown in the next section, and \eqref{class_va}  is shown in Sec. \ref{sec_c}.
\end{proof}

\begin{remark}
For the classical parametrization, we are  limited to discussing the scalar parameter case, due to `$\mathcal{M}\circ  f$' is not well-defined   
\end{remark}

\subsubsection{Equilibrate parametrization}
Here, we derive the result   \eqref{equ_para}, which reveals that:
preconditioning the  constraint in \eqref{pri1} is equivalent to selecting the variable metric.

Let $ \mathcal S  \in \mathfrak B (\mathscr H) $ be bijective.  \eqref{pri1} can be rewritten into
\begin{align}\label{pro_right}
&\underset{\bm x,\bm z}{\text{minimize}}\quad   f(\bm x) + g(\bm z) \qquad\qquad\qquad\qquad\quad\,\,\,\,\quad\,\, \text{subject\,to}\quad  \bm{A}\bm x - \bm{B}\bm z =  \bm c  \nonumber\\
\iff
&\underset{\bm x,\bm z}{\text{minimize}}\quad    f(\mathcal{S}\bm{A})^{-1}(\mathcal{S}\bm{A}\bm x) + g(\mathcal{S}\bm{B})^{-1}(\mathcal{S}\bm{B}\bm z) \quad \text{subject\,to}\quad  \mathcal{S}\bm{A}\bm x - \mathcal{S}\bm{B}\bm z = \mathcal{S}\bm c\nonumber\\
\iff
&\underset{\widetilde{\bm x},\widetilde{\bm z}}{\text{minimize}}\quad    \widetilde{f}(\widetilde{\bm x}) + \widetilde{g}(\widetilde{\bm z}) \qquad\qquad\qquad\qquad\qquad\quad \text{subject\,to}\quad   \widetilde{\bm x} - \widetilde{\bm z} = \widetilde{\bm c} .
\end{align} 
with variable substitutions:
\begin{equation}
\widetilde{\bm x} =  \mathcal{S}\bm{A}\bm x, \quad  \widetilde{\bm z} =  \mathcal{S}\bm{B}\bm z , \quad \widetilde{\bm c} = \mathcal{S}\bm c,
\end{equation}
and function substitutions:
\begin{equation}
\widetilde{f} = f\circ(\mathcal{S}\bm{A})^{-1}, \quad \widetilde{g} = g\circ(\mathcal{S}\bm{B})^{-1}.
\end{equation}
The Lagrangian \eqref{or_L} can be rewritten into 
\begin{align}
\mathcal{L}(\bm{x},\bm{z},\bm{\lambda}) 
\,=\,\,	&  f(\bm{x}) + g (\bm{z})  + \langle   \bm{\lambda}, \, \bm{A}\bm{x} - \bm{B}\bm{z} - \bm{c}  \rangle   \nonumber\\
\,=\,\,	&  f(\bm{x}) + g (\bm{z})  + \big\langle   (\mathcal{S}^*)^{-1}  \bm{\lambda}, \, \mathcal{S} \big( \bm{A}\bm{x} - \bm{B}\bm{z} - \bm{c}  \big)\big\rangle  \nonumber\\
\,=\,\,	& \underbrace{\,\widetilde{f}(\widetilde{\bm{x}}) + \widetilde g (\widetilde{\bm{z}})  + \langle \widetilde  {\bm\lambda}, \,  \widetilde{\bm{x} }- \widetilde{\bm{z}} - \widetilde{\bm c}\rangle\,}_{
	\eqdef \,\,\, \widetilde{\mathcal{L}}(\widetilde{\bm x}, \widetilde{\bm z}, \widetilde{\bm\lambda})} ,
\end{align}
with a dual-variable substitution:
\begin{equation}
\widetilde  {\bm\lambda} \eqdef  (\mathcal{S}^*)^{-1}  \bm{\lambda}.
\end{equation}
The Fenchel dual problem is  found via
\begin{align}\label{new_dual}
 \underset{\bm\lambda}{\text{maximize}} \,\,\, \underset{{\bm x}, {\bm z}}{\inf}\,\, \mathcal{L}({\bm x}, {\bm z}, \bm\lambda) 
=\,\, & \underset{\widetilde{\bm\lambda}}{\text{maximize}} \,\,\, \underset{\widetilde{\bm x}, \widetilde{\bm z}}{\inf}\,\, \widetilde{\mathcal{L}}(\widetilde{\bm x}, \widetilde{\bm z}, \widetilde{\bm\lambda}) \nonumber\\
=\,\, &  \underset{\widetilde{\bm\lambda}}{\text{maximize}} \,\, 
- \bigg( \underset{\widetilde{\bm x}}{\sup}\, \biggl\{\langle -\widetilde{\bm\lambda}, \widetilde{\bm x} \rangle - \widetilde{f}(  \widetilde{\bm x} ) \biggr\}  	
+\underset{\widetilde{\bm z}}{\sup}\, \biggl\{ \langle \widetilde{\bm\lambda}, \widetilde{\bm z} \rangle - \widetilde{g}   ( \widetilde{\bm z}  )\biggr\}
+ \langle\widetilde{\bm\lambda},\, \widetilde {\bm c} \rangle 		\bigg) \nonumber\\
=\,\, &   \underset{\widetilde{\bm\lambda}}{\text{maximize}} \,\, 
- \bigg(  \widetilde{f} ^* (-\widetilde{\bm\lambda})	+   \widetilde{g}^*(\widetilde{\bm\lambda}) 	+ \langle\widetilde{\bm\lambda},\, \widetilde{\bm c} \rangle  \bigg),
\end{align}
which can be written into a minimization form
\begin{equation}
\underset{\widetilde{\bm\lambda}}{\text{minimize}} \,\,\,\,  \widetilde{f} ^* (-\widetilde{\bm\lambda})	+   \widetilde{g}^*(\widetilde{\bm\lambda}) 	+ \langle\widetilde{\bm\lambda},\, \widetilde{\bm c} \rangle .
\end{equation}
Invoking the following  lemma gives the result \eqref{equ_para} in Proposition \ref{prop_uni_pd}.
\begin{lem}
	Given  function $ g  \in \Gamma_0(\mathscr  H) $, denote a post-composition as 	$ h (\cdot ) \eqdef g (\cdot) +  \langle \bm{c}, \cdot \rangle$. Then,
	\begin{equation}
	h^* (\cdot) = g^*(\cdot - \bm{c})
	\end{equation}
\end{lem}
\begin{proof}
	By the Fenchel conjugate definition, 
	\begin{equation}
	h^* (\bm v) = \underset{\bm{z}}{\sup}\, \langle \bm{z}, \bm v \rangle - g (\bm{z}) - \langle \bm{c}, \bm{z} \rangle  
	= \underset{\bm{z}}{\sup}\, \langle \bm{z}, \bm v - \bm{c}\rangle - g (\bm{z}) =  g^* (\bm{v} - \bm{ c}),
	\end{equation}	
	which concludes the proof.
\end{proof}

\begin{remark}[constraint preconditioning]
	In view of \eqref{pro_right}, we see that the ADMM constraint becomes 
	\begin{equation}
	\widetilde{\bm x} - \widetilde{\bm z} = \widetilde{\bm c} \,\, \iff\,\,  \mathcal{S}\big(\bm{A}\bm x - \bm{B}\bm z - \bm c\big)   = 0 . 
	\end{equation}
	This  coincides with a preconditioning on the constraint, recall the preconditioning  from \eqref{precon}. That said,  preconditioning the ADMM constraint is equivalent to  selecting $ \mathcal{S} $. 
\end{remark}

\subsubsection{Classical parametrization}\label{sec_c}
Here, we derive the result   \eqref{class_va}, and the process reveals that:
scaling the objective function is equivalent to  selecting scalar step-size.

Let $\gamma>0$.  \eqref{pri1} can be rewritten into
\begin{align}\label{for0}
&\underset{\bm x,\bm z}{\text{minimize}}\quad   f(\bm x) + g(\bm z)\qquad\qquad\qquad\qquad\quad\text{subject\,to}\quad  \bm{A}\bm x - \bm{B}\bm z =  \bm c  \nonumber\\
\iff
&\underset{\bm x,\bm z}{\text{minimize}}\quad    \frac{1}{\gamma}\biggl( f\bm{A}^{-1}(\bm{A}\bm x) +  g\bm{B}^{-1}(\bm{B}\bm z)\biggr) \quad \text{subject\,to}\quad  \bm{A}\bm x -\bm{B}\bm z = \bm c\nonumber\\
\iff
&\underset{\widetilde{\bm x},\widetilde{\bm z}}{\text{minimize}}\quad    \widetilde{f}(\widetilde{\bm x}) + \widetilde{g}(\widetilde{\bm z}) \qquad\qquad\qquad\qquad\quad\, \text{subject\,to}\quad   \widetilde{\bm x} - \widetilde{\bm z} = {\bm c} .
\end{align}
with variable substitutions:
\begin{equation}
\widetilde{\bm x} =  \bm{A}\bm x, \quad  \widetilde{\bm z} = \bm{B}\bm z 
\end{equation}
and function substitutions:
\begin{equation}
\widetilde{f} =  \frac{1}{\gamma}f\circ\bm{A}^{-1}, \quad \widetilde{g} = \frac{1}{\gamma}g\circ\bm{B}^{-1}.
\end{equation}
The Lagrangian \eqref{or_L} can be rewritten into 
\begin{align}
\mathcal{L}(\bm{x},\bm{z},\bm{\lambda}) 
\,=\,\,	&  f(\bm{x}) + g (\bm{z})  + \langle   \bm{\lambda}, \, \bm{A}\bm{x} - \bm{B}\bm{z} - \bm{c}  \rangle   \nonumber\\
\,=\,\,	&  \gamma\cdot \frac{1}{\gamma}\biggl(f(\bm{x}) + g (\bm{z})  + \langle   \bm{\lambda}, \, \bm{A}\bm{x} - \bm{B}\bm{z} - \bm{c}  \rangle   \biggr)\nonumber\\
\,=\,\,	&  \gamma\,\, \underbrace{\biggl(\widetilde{f}(\widetilde{\bm{x}}) + \widetilde g (\widetilde{\bm{z}})  + \langle \widetilde  {\bm\lambda}, \,  \widetilde{\bm{x} }- \widetilde{\bm{z}} - \bm c\rangle\biggr)}_{\eqdef \,\,
\widetilde{\mathcal{L}}(\widetilde{\bm x}, \widetilde{\bm z}, \widetilde{\bm\lambda})}
\end{align}
with a  dual-variable substitution:
\begin{equation}
\widetilde  {\bm\lambda} \eqdef  \frac{1}{\gamma} \bm{\lambda}.
\end{equation}
The Fenchel dual problem is  found via
\begin{align}
\underset{\bm\lambda}{\text{maximize}} \,\, \underset{{\bm x}, {\bm z}}{\inf}\,\, \mathcal{L}({\bm x}, {\bm z}, \bm\lambda) 
=\,\, & \underset{\widetilde{\bm\lambda}}{\text{maximize}} \,\, \underset{\widetilde{\bm x}, \widetilde{\bm z}}{\inf}\,\,\,\gamma\, \widetilde{\mathcal{L}}(\widetilde{\bm x}, \widetilde{\bm z}, \widetilde{\bm\lambda}) \nonumber\\
=\,\, & \underset{\widetilde{\bm\lambda}}{\text{maximize}} \,\, \underset{\widetilde{\bm x}, \widetilde{\bm z}}{\inf}\,\,\,\widetilde{\mathcal{L}}(\widetilde{\bm x}, \widetilde{\bm z}, \widetilde{\bm\lambda}).
\end{align}
The rest procedure  follows exactly the same as  in  \eqref{new_dual}, which yields \eqref{class_va}.

\begin{remark}[data/function scaling]
In view of \eqref{for0}, we see that the ADMM objective function is scaled into
\begin{equation}
\frac{1}{\gamma}\biggl( f(\bm x) + g(\bm z) \biggr)
\end{equation}
This  coincides with the function preconditioning, recall   from \eqref{precon}. 
\end{remark}

\section{Optimal metric/preconditioner selection}\label{sec_opt_met}
In the previous section, we establish the equivalence between the metric selection and preconditioning. Here, we will find their optimal choice via our \textit{equilibrate upper-bound minimization} technique. One can employ any unscaled fixed-point in principle. Here, we demonstrate the ADMM case.

\subsection{Optimization problem}\label{sec_i} 
The optimal  metric selection problem  is straightforward by invoking our \textit{equilibrate upper-bound minimization} technique. Here, we present an explicit expression for the ADMM case, with  the unscaled fixed-point established in \eqref{un_fp}.
\begin{align}\label{to_opt}
 &	\underset{\mathcal S  }{\text{minimize}}\,\,  	 \Vert \mathcal S\bm{A}\bm{x}^\star +  (\mathcal{S}^*)^{-1} \bm{\lambda}^\star  - \bm{\zeta}^0  \Vert^2 \nonumber\\
=&  \underset{\mathcal S  }{\text{minimize}}\,\,  	 \Vert \mathcal S\bm{Ax}^\star\Vert^2  +  \Vert (\mathcal{S}^*)^{-1} \bm{\lambda}^\star    \Vert^2 - 2 \langle\mathcal S\bm{A}\bm{x}^\star, \bm{\zeta}^0 \rangle -    2 \langle(\mathcal{S}^*)^{-1} \bm{\lambda}^\star,  \bm{\zeta}^0 \rangle
 ,  \tag{opt.  sel.}
\end{align}
with $ \mathcal S  \in \mathfrak B (\mathscr H) $  being bijective, where $  \bm{\zeta}^0 $   is an arbitrary initialization. 

The above problem  is  not  readily solvable, due to 
operator $ \mathcal S $ has too much degree of freedom. 
To proceed, we need to specify $ \mathcal S  $,  i.e., a construction issue.
Recall our goal is to improve efficiency. Then, ideally, the construction procedure  should satisfy that

\vspace{6pt}

$\bullet$ (i)  Algorithm iterates admit closed-form expressions;
\vspace{6pt}

$\bullet$  (ii) Problem  \eqref{to_opt}  admits a closed-form solution.
\vspace{6pt}

In the following part, we will propose an ideal construction tailored to $l_1$-minimization. Quite remarkably, we prove a one-iteration convergence property,  which in turn  verifies its optimality and also the power of  our \textit{equilibrate upper-bound minimization} technique. 


\subsection{Tailored  construction for $ l_1  $-norm}
For the sake of clarity, we start with the basic  case. The same arguments are applicable  to a  more  general  formulation,  discussed at the end of the section. 
 \begin{align}\label{basic0}
 \underset{\bm{x}}{\text{minimize}}\quad   f(\bm{x}) + \alpha\Vert \bm{x} \Vert_1,  \tag{basic}
 \end{align}
 with  $ f \in \Gamma_0 (\mathscr H)  $  being smooth, and  $\alpha > 0$ being a   regularization parameter.
 It can be rewritten into the following ADMM form:
\begin{align}\label{l1}
&\underset{\bm{x},\bm{z}}{\text{minimize}}\quad   f(\bm{x}) + \alpha\Vert \bm{z} \Vert_1 \nonumber\\
&\text{subject\,to}\quad \,\, \bm{x} = \bm{z}.\tag{$ l_1 $-norm}
\end{align}
The  augmented Lagrangian in a metric environment $\mathscr H_\mathcal{M}$  is 
\begin{equation}\label{lagr}
\mathcal{L}_\mathcal{M}(\bm{x},\bm{z},\bm{\lambda}) 
\,=\,  f(\bm{x}) + \alpha \Vert   \bm{z} \Vert_1 + \frac{1}{2}\Vert  \bm{x} - \bm{z}  + \mathcal{M}^{-1}\bm{\lambda}\Vert^2_\mathcal{M},  
\end{equation}

\subsubsection{Closed-form evaluation}
First, we address the ADMM iterates closed-form issue, as in Sec. \ref{sec_i}. This is  challenging  due to the non-smoothness of the  $l_1  $-norm function. In the literature, the following problem is considered no closed-form solution available in a non-Euclidean space:
\begin{equation}
\underset{\bm{x}}{\text{argmin}} \,\, \Vert \bm{x} \Vert_1 + \frac{1}{2}\Vert\bm{x} - \bm{v} \Vert^2_\mathcal{M}. 
\end{equation}
On one hand, this is true if $ \mathcal{M} $ is an arbitrary (positive definite) operator. On the other hand, we will show that there  exists  a special class of metrics, which turns out surprisingly powerful.

To start, we need two lemmas.
\begin{lem}
	The sub-differential operator of $l_1$-norm can be characterized as
	\begin{equation}\label{sub_cha}
	\partial \Vert \bm{x} \Vert_1 = \{ \text{sgn}(\bm{x}) + \bm{w} \,\vert\,	\bm{x}\odot \bm{w} = 0, \,\Vert \bm{w} \Vert_\infty \leq 1 \}, 
	\end{equation}	
	where by $ \odot $ we denote the Hadamard/element-wise product,  by  `sgn' the following sign function:
	\begin{equation}
	\text{sgn}(a) 
	= \begin{cases}
	-1  &   a < 0, \\
	0 & 	a = 0, \\
	1 & 	a > 0.
	\end{cases} 
	\end{equation}
\end{lem}
\begin{proof}
	Recall the following general characterization for any norm:
	\begin{equation}\label{ge_1}
	\partial \Vert \bm{x} \Vert = \{  \bm{v} \,\vert\,	\langle \bm{v} , \, \bm{x}   \rangle	 = \Vert \bm{x} \Vert, \, \Vert \bm{v} \Vert_* \leq 1 \},
	\end{equation}	
	where $ \Vert \cdot \Vert_* $ denotes the dual norm, which corresponds to $ \Vert \cdot \Vert_\infty $ in our case.
	We aim to show that our characterization satisfies this  general condition. 
	
	Let $ \bm{v} =  \text{sgn}(\bm{x}) + \bm{w}$, and we obtain 
	\begin{equation}
	\langle \bm{v} , \, \bm{x}   \rangle = \langle \text{sgn}(\bm{x}) + \bm{w} , \, \bm{x}   \rangle = \langle \text{sgn}(\bm{x})  , \, \bm{x}   \rangle = \Vert \bm{x} \Vert_1. 
	\end{equation}
	Also, we will have
	\begin{equation}
	\Vert \bm{v} \Vert_\infty = \Vert \text{sgn}(\bm{x}) + \bm{w} \Vert_\infty \leq 1, 
	\end{equation}
	due to that 
	\begin{equation}
	\Vert \text{sgn}(\bm{x}) \Vert_\infty \leq 1,\quad  \Vert\bm{w} \Vert_\infty \leq 1 ,\quad \bm{x}\odot \bm{w} = 0.
	\end{equation}	
	The proof is now concluded.
\end{proof}	
\begin{lem}\label{fact}
	The following relation holds:
\begin{equation}
\text{sgn}(\bm{x}) = \text{sgn}(\bm{z}) \quad \iff \quad \partial \Vert \bm{x} \Vert_1 = \partial \Vert \bm{z} \Vert_1 .
\end{equation}
\end{lem}
\begin{proof}
	Recall the the sub-differential characterization as in \eqref{sub_cha} 
	\begin{equation}
	\partial \Vert \bm{x} \Vert_1 = \{ \text{sgn}(\bm{x}) + \bm{w} \,\vert\,	\bm{x}\odot \bm{w} = 0, \,\Vert \bm{w} \Vert_\infty \leq 1 \}, 
	\end{equation}

	First, consider the collection of all non-zero elements of $ \bm{x} $, denoted by $ \{ x_j \,|\,  x_j \neq 0,\, \forall j\} $. By $ 	\bm{x}\odot \bm{w} = 0 $, we have the corresponding elements of  $ \bm{w}  $, i.e., $ w_j = 0, \forall j $. Then
	Then, the above can be written as
	\begin{equation}
	\text{sgn}(x_j)  = \partial \Vert x_j \Vert_1 ,
	\end{equation}
	and hence if $ \text{sgn}(x_j) = \text{sgn}(z_j) $, then $  \partial \Vert x_j \Vert_1 =  \partial \Vert z_j \Vert_1 $.
	
	Now, consider the collection of all zero elements of  $ \bm{x} $, denoted by $ \{ x_l \,|\,  x_l = 0, \, \forall l \} $. Then, 
	\begin{equation}
	\partial \Vert x_l \Vert_1 = [-1, 1],
	\end{equation}
	which is the same for any zero input, i.e., if $ \text{sgn}(x_l) = \text{sgn}(z_l) $, then
	\begin{equation}
		\partial \Vert x_l \Vert_1 = [-1, 1] = \partial \Vert z_l \Vert_1. 
	\end{equation}
	The proof is therefore concluded.
\end{proof}

Now, we are ready to present the closed-form evaluation result for a certain
 class of   metrics.
	\begin{theo}\label{prop_evaluation}
	Consider a variable metric space $\mathscr H_\mathcal{M}$,  with norm $ \Vert\cdot \Vert_\mathcal{M} $ induced by the inner product $ \langle \cdot, \,  \cdot \rangle_\mathcal{M} = \langle \cdot, \,  \mathcal{M}\,\cdot \rangle  $.
	Suppose the variable metric $ \mathcal{M} $ admits the following element-wise characterization:
	\begin{equation}\label{assu1}
	\mathcal{M}^{-1} (\bm{v}) = \bm{m} \odot \bm{v},
	\end{equation}
	with $ \bm{m} \in \mathscr R_{++} $ an arbitrary positive vector.	
	Then, 
	\begin{equation}\label{prox_l1}
	\mathcal{M}^{-1}\circ\mathcal{T}\circ\mathcal{M} \,\,  (\cdot) 
	=  \underset{\bm{x}}{\text{argmin}} \,\, \Vert \bm{x} \Vert_1 + \frac{1}{2}\Vert\bm{x} - \cdot \Vert^2_\mathcal{M}  ,
	\end{equation}
	where  $ \mathcal{T}  $ denotes the  soft-thresholding operator:
	\begin{equation}\label{soft}
	\mathcal{T} (\bm{u}) \eqdef \text{sgn}(\bm{u}) \odot \big\{|\bm{u}| -\bm{1}\big\}_+, 
	\end{equation}	
	 by  $ |\cdot| $  the absolute value, by $ \bm{1} $ the ones vector, by $ \{\cdot\}_+ $  the projection onto the non-negative orthant.
\end{theo}

\begin{proof}
Our goal is to show  \eqref{prox_l1}.  By the first-order optimality condition, it can be rewritten into
\begin{equation}\label{1_norm}
\mathcal{M}\,(\bm{v} ) -\mathcal{T}\mathcal{M} \,(\bm{v} )  \in \partial\Vert\mathcal{M}^{-1}\circ\mathcal{T}\circ\mathcal{M} \,\,(\bm{v} ) \Vert_1. \tag{goal}
\end{equation}

To start,  define the following orthogonal decomposition:
\begin{equation}
\mathcal{M}\,(\bm{v} )  =  \underbrace{\big\{ \mathcal{M}\bm{v}  \big\}_{\Vert\cdot\Vert_\infty>1}}_{\bm{y}_1 }  
+ \underbrace{\big\{ \mathcal{M}\bm{v}  \big\}_{\Vert\cdot\Vert_\infty\leq 1}}_{\bm{y}_2 }  
\eqdef
\bm{y}_1 + \bm{y}_2,
\end{equation}
which implies
\begin{equation}\label{prop}
\Vert\bm{y}_1\Vert_\infty > 1, 
\quad \Vert\bm{y}_2\Vert_\infty \leq 1, 
\quad\bm{y}_1 \odot \bm{y}_2 = \bm{0}. \tag{properties}
\end{equation}
Moreover, by definition \eqref{soft} of the  soft-thresholding operator  $ \mathcal{T} $, we have
\begin{equation}
\mathcal{T}(\bm{y}_2) = 0.
\end{equation}
It follows that
\begin{equation}
\mathcal{T}\circ \mathcal{M}\,(\bm{v} ) \,=\, \mathcal{T}\,(\bm{y}_1 + \bm{y}_2 ) \,=\,  \text{sgn}(\bm{y}_1) \odot (|\bm{y}_1| -\bm{1})
\,=\,  \bm{y}_1  - \text{sgn}(\bm{y}_1).
\end{equation}
In view of \eqref{1_norm}, its  left-hand side can now be rewritten into
\begin{equation}\label{rec_1}
\mathcal{M}\,(\bm{v} ) -\mathcal{T}\mathcal{M} \,(\bm{v} ) \,=\, \bm{y}_1 + \bm{y}_2  - \big(\bm{y}_1  - \text{sgn}(\bm{y}_1) \big)
\,=\, \text{sgn}(\bm{y}_1) + \bm{y}_2,
\end{equation}
which is an element of the set $ \partial \Vert  \bm{y}_1 \Vert_1 $,  i.e.,
\begin{equation}
\mathcal{M}\,(\bm{v} ) -\mathcal{T}\mathcal{M} \,(\bm{v} ) = \text{sgn}(\bm{y}_1) + \bm{y}_2  \in \partial \Vert \bm{y}_1 \Vert_1,
\end{equation}
which holds  owing to the sub-differential characterization in \eqref{sub_cha}, and \eqref{prop}.
Comparing  the above to \eqref{1_norm}, all what left is to show 
\begin{equation}
\partial \Vert \bm{y}_1 \Vert_1 = \partial \Vert 	\mathcal{M}^{-1}\circ\mathcal{T}\circ\mathcal{M} \,\,(\bm{v} )  \Vert_1 = \partial \Vert \bm{m}\odot  \big(\bm{y}_1  - \text{sgn}(\bm{y}_1) \big) \Vert_1. \tag{goal 2}
\end{equation}
where the last equality uses the special metric \eqref{assu1}. 
Owing to Lemma \ref{fact}, this is equivalent to showing 
\begin{equation}\label{to_s1}
\text{sgn}(\bm{y}_1) = \text{sgn}\bigg(   	\bm{m}\odot  \big(\bm{y}_1  - \text{sgn}(\bm{y}_1) \big)  \bigg).
\end{equation}
To show this, from its right-hand side, we have
\begin{align}
\text{sgn}\bigg(   	\bm{m}\odot  \big(\bm{y}_1  - \text{sgn}(\bm{y}_1) \big)  \bigg)
&= \text{sgn}\bigg( \bm{m}\odot \text{sgn}(\bm{y}_1) \odot (|\bm{y}_1| - \bm{1})\bigg)  \nonumber\\
&= \text{sgn}(\bm{y}_1) \odot \text{sgn}\bigg( \bm{m}\odot  (|\bm{y}_1| - \bm{1})\bigg) \nonumber\\
&= \text{sgn}(\bm{y}_1),
\end{align}
where the last line holds due to $ (|\bm{y}_1| - \bm{1}) $  and $ \bm{m} $ are positive vectors, which implies
\begin{equation}
\text{sgn}\bigg( \bm{m}\odot  (|\bm{y}_1| - \bm{1})\bigg) = \bm{1}.
\end{equation}
The proof is therefore concluded.
\end{proof}

\subsection{Optimal choice}
In the above content, we construct a  class of metrics to guarantee closed-form  iterates. Here, we would apply  them to  the problem  \eqref{to_opt}.  
We will obtain  a  closed-form optimal choice.
Moreover,  the  ultimate one-iteration convergence will be achieved under zero initialization. We  therefore limit our discussion  to  zero initialization, since other initializations will bring additional, unnecessary  complications, see formulation \eqref{to_opt}.

	\begin{prop}\label{prop_matrix_step}
		Consider a variable metric space $\mathscr H_\mathcal{M}$ with $ \mathcal{S} $ a decomposed metric as in Lemma \ref{lem_emb},  i.e.,  $ \langle \cdot, \,  \mathcal{M}\,\cdot \rangle = \langle \mathcal{S}\, \cdot, \,  \mathcal{S}\,\cdot \rangle $.
			Suppose metric $ \mathcal{M} $ admits the following element-wise characterization:
		\begin{equation}\label{metri_op}
		\mathcal{M}^{-1} (\bm{v}) = \bm{m} \odot \bm{v}.
		\end{equation}
			with $ \bm{m} \in \mathscr R_{++} $ an arbitrary positive vector.
		
	Then, under zero-initialization, the optimal choice can be generally written as
	\begin{equation}\label{metric_op}
	\mathcal{M^\star}^{-1} = \text{abs}\bigg(  {\bm{x}^\star} \oslash {\bm{\lambda}^\star}   \bigg) \, \iff \, \mathcal{M^\star}= \text{abs}\bigg(  {\bm{\lambda}^\star} \oslash {\bm{x}^\star}   \bigg),
	\end{equation}
	where  by $ \oslash $ we denote the Hadamard/element-wise division, 
	by `$ \text{abs} $' the element-wise absolute value ,  by $ \star $ the optimal solution.
	Particularly, if the $ i$-th  element $  x_i^\star = 0 $,  we choose
	\begin{equation}
	\text{abs}\bigg({x_i^\star }/{\lambda_i^\star} \bigg)  \downarrow 0, \qquad  \text{abs}\bigg({\lambda_i^\star}/{x_i^\star } \bigg)  \rightarrow + \infty	 .
	\end{equation}
	Furthermore, if $x_i^\star =  \lambda_i^\star = 0$, then all feasible (positive) choices are equivalent.
\end{prop}

\begin{proof}
	Consider zero initialization, \eqref{to_opt} can be simplified into
\begin{align}\label{ob}
\underset{\mathcal S  }{\text{minimize}}\,\,  	 \Vert \mathcal S\bm{x}^\star\Vert^2  +  \Vert (\mathcal{S}^*)^{-1} \bm{\lambda}^\star    \Vert^2 .
\end{align}

To solve  the above problem, first we decompose the metric \eqref{metri_op}.
For the positive vector $ \bm{m} $,  denote its non-zero decomposition as
\begin{equation}
\bm{m}  = \bm{h}  \odot \bm{h} \, \iff \, 	\mathcal{S}^{-1} (\bm{v}) = 	{\mathcal{S}^*}^{-1} (\bm{v}) = \bm{h} \odot \bm{v} \,\iff\, 	\mathcal{S} (\bm{v}) =  \bm{v} \oslash \bm{h} .
\end{equation}
with $ \bm{h} \in \mathscr R / \{0\}$. Clearly, the above defined $ \mathcal{S} $ is bijective, bounded and linear.
Then, following the goal in \eqref{ob}, we can rewrite its objective function into 
\begin{align}	 \Vert \mathcal S\bm{x}^\star\Vert^2  +  \Vert (\mathcal{S}^*)^{-1} \bm{\lambda}^\star    \Vert^2 
=\, &    	 \Vert\bm{x}^\star \oslash \bm{h}\Vert^2  +  \Vert  \bm{h}  \odot  \bm{\lambda}^\star    \Vert^2 \\
=\, &   \sum_{i}\vert   x^\star_i / h_i  \vert^2 +  \sum_{i} \vert h_i  \lambda^\star_i  \vert^2 \nonumber\\
=\, &  \sum_{i} \vert x^\star_i\vert^2/m_i +  \sum_{i} m_i \vert{\lambda}^\star_i \vert^2, 
\end{align}
Due to the element-wise operation, the optimal metric selection problem is separable. For the $ i $-th problem, we have
\begin{equation}
\underset{m_i>0}{\text{argmin}} \,\, \vert{x}^\star_i\vert^2/m_i  +  m_i\vert{\lambda}^\star_i \vert^2.
\end{equation}
Suppose $x_i^\star =  \lambda_i^\star = 0$, the choice of $ m_i^\star  $ is can take any  positive value.
Suppose only $  x^\star_i = 0 $, the optimal choice is by setting 
\begin{equation}
{m}^\star_i \downarrow 0,
\end{equation}
and the objective value will be arbitrarily close to 0.
Otherwise, the optimal solution is given by
\begin{equation}
m_i^\star = \frac{\vert  x^\star_i\vert}{\vert{\lambda}^\star_i\vert}. 
\end{equation}
The proof is now concluded.

\end{proof}

Following the above definition, it is worth mentioning a later useful property.
\begin{lem}
Let us note that the metric $\mathcal{M}$ is a non-zero/bijective mapping. For the  element-wise operation in \eqref{metric_op},
it is safe to treat 
\begin{equation}
\bm{a} \oslash {\bm{b}}  \odot {\bm{b}}  \,\,=\,\, \bm{a}  \odot {\bm{b}} \oslash {\bm{b}}  \,\,=\,\,  {\bm{a}}.
\end{equation}
\end{lem}

\subsection{Provable one-iteration convergence}\label{sec_one_itr}
Here, we employ the optimal  metric choice to solve the problem \eqref{l1}, which yields a closed-form solution.
It may worth noticing that the result is not limited to ADMM, since it is only a  proximal evaluation step.

\begin{claim}\label{claim}
Selecting the  variable metric $\mathcal{M}$ via Proposition \ref{prop_matrix_step}. Then, under zero-initialization,
\begin{equation}
\bm{x}^{1} 
=\, \underset{\bm{x}}{\text{argmin}} \, 	\mathcal{L}_\mathcal{M}(\bm{x},0,0)
=\, \underset{\bm{x}}{\text{argmin}} \,\, f (\bm{x} ) + \frac{1}{2}\Vert\bm{x} - 0 \Vert^2_\mathcal{M}, 
\end{equation}  
is an optimal solution to 
\begin{align}\label{l1_pro}
&\underset{\bm{x},\bm{z}}{\text{minimize}}\quad   f(\bm{x}) + \alpha\Vert \bm{z} \Vert_1 \nonumber\\
&\text{subject\,to}\quad\,\,  \bm{x} = \bm{z},\tag{$ l_1 $}
\end{align}
with the augmented Lagrangian defined as
\begin{equation}\label{lagr2}
\mathcal{L}_\mathcal{M}(\bm{x},\bm{z},\bm{\lambda}) 
\,=\,  f(\bm{x}) + \alpha \Vert   \bm{z} \Vert_1 + \frac{1}{2}\Vert  \bm{x} - \bm{z}  + \mathcal{M}^{-1}\bm{\lambda}\Vert^2_\mathcal{M},  
\end{equation}

\end{claim}

To show this, we need 2 lemmas.
%

\begin{lem}\label{lem02}
The dual solution $ \bm \lambda^\star $ of problem \eqref{l1_pro} satisfies
\begin{equation}
 \bm \lambda^\star \in \alpha \,\partial \Vert  \bm{x}^\star  \Vert_1
\end{equation}
\end{lem}
\begin{proof}
	In view of problem \eqref{l1_pro}, the primal-dual solution, or equivalently,
	the saddle point is given by
	\begin{equation}
	(\bm{x}^\star,\bm{z}^\star,\bm{\lambda}^\star) = \, \underset{\bm{\lambda}^\star}{\sup}\, \underset{\bm{x}^\star, \bm{z}^\star}{\inf} \, \mathcal{L}_\mathcal{M}(\bm{x},\bm{z},\bm{\lambda}) 
	\end{equation}
	with  $ \mathcal{L}_\mathcal{M} $ given in \eqref{lagr2}. It follows that
\begin{align}
 \underset{\bm{z}}{\inf} \,\,  \alpha\Vert \bm{z}^\star \Vert_1 + \langle \bm{\lambda}^\star, \,  - \bm{z}^\star \rangle 
 \iff 
  \bm{\lambda}^\star \in  \alpha \, \partial \Vert \bm{z}^\star  \Vert_1  
\end{align}
By the problem constraint $ \bm{x} = \bm{z}  $, the saddle point satisfies
\begin{equation}
\bm{x}^\star = \bm{z}^\star
\end{equation}
which concludes the proof.
\end{proof}

\begin{lem}\label{lem03}
	The following holds:
	\begin{equation}
	\text{sgn}(\bm{x}) \,\odot  \text{abs}\big(  \partial \Vert  \bm{x}  \Vert_1	\big) = \text{sgn}(\bm{x}) ,
	\end{equation}
	where abs$ (\cdot) $ denotes the element-wise absolute value operation.
\end{lem}
\begin{proof}
	First, consider the collection of  non-zero elements of $ \bm{x} $, denoted by $ \{ x_j \,|\,  x_j \neq 0,\, \forall j\} $. Then, the above can be written into
	\begin{equation}
	\text{sgn}(x_j) \cdot  \text{abs}\big( \text{sgn}(x_j) \big) = \text{sgn}(x_j) .
	\end{equation}
	Now, consider the collection of  zero elements of  $ \bm{x} $, denoted by $ \{ x_l \,|\,  x_l = 0, \, \forall l \} $. Then, the above can be written as
	\begin{equation}
	\text{sgn}(x_l)  \cdot\text{abs}\big( \text{sgn}(x_l) \big) = 0 \cdot\text{abs}\big( \text{sgn}(x_l) \big) =  \,0\, =  \text{sgn}(x_l).
	\end{equation}
	The proof is therefore concluded.
\end{proof}

Now, we are ready to show Claim \ref{claim}.
\begin{theo}
Selecting metric $\mathcal{M}$ via Proposition \ref{prop_matrix_step}. Then,
\begin{equation}
\bm{x}^1 =   \underset{\bm{x}}{\text{argmin}} \,\, f(\bm{x}) + \frac{1}{2}\Vert \bm{x}\Vert^2_{\mathcal{M}},
\end{equation}
is an optimal solution to  problem \eqref{l1_pro}.
\end{theo}

\begin{proof}
In view of  problem \eqref{l1_pro}, an iterate $ \bm{x}^k $ is a solution if and only if	
\begin{equation}\label{basic}
0\, \in \,    \nabla f(\bm{x}^k ) + \alpha\, \partial \Vert  \bm{x}^k   \Vert_1.
\end{equation}	
	We will show that the above holds under our choice of metric.

	By definition, the following always holds:
	\begin{align}\label{rel0}
	\bm{x}^1 =   \underset{\bm{x}}{\text{argmin}} \,\, f(\bm{x}) + \frac{1}{2}\Vert \bm{x}\Vert^2_{\mathcal{M}} 
	\iff
	 0\, = \,    \nabla f(\bm{x}^1) +  \mathcal{M}  (\bm{x}^1). \tag{cond.}
	\end{align}

	To start, suppose  
	\begin{equation}
	\bm{x}^1 =   \bm{x}^\star. 
	\end{equation}
	Then, \eqref{rel0} implies that
	\begin{align}
	0\, = \,    \nabla f(\bm{x}^1) + \mathcal{M} (\bm{x}^1)  
	\iff
	0\, &= \,    \nabla f(\bm{x}^1) + \mathcal{M} (\bm{x}^\star)  \nonumber\\
		\iff
	0\, &= \,    \nabla f(\bm{x}^1) +   	\text{abs}\big(   \bm{\lambda}^\star \oslash \bm{x}^\star 	\big) \odot \bm{x}^\star \nonumber\\
	\iff
	0\, &= \,    \nabla f(\bm{x}^1) +    \text{abs}\big(   \bm{\lambda}^\star 	\big) \odot \text{sgn}(\bm{x}^\star)		\nonumber\\
	\iff
	0\, &\in \,    \nabla f(\bm{x}^1) +   \text{abs}\big( \alpha\, \partial \Vert  \bm{x}^\star \Vert_1 \big)  \odot \text{sgn}(\bm{x}^\star) \tag{via Lemma \ref{lem02}}\\
	\iff
	0\, &\in \,    \nabla f(\bm{x}^1) + \alpha\, \text{sgn}(\bm{x}^\star)  \tag{via Lemma \ref{lem03}}\\
	\iff
	0\, &\in \,    \nabla f(\bm{x}^1) + \alpha\, \partial \Vert  \bm{x}^\star \Vert_1 \nonumber\\
	\iff
	0\, &\in \,    \nabla f(\bm{x}^1) + \alpha\, \partial \Vert  \bm{x}^1 \Vert_1,
	\end{align}
	which meets condition \eqref{basic}. That is,   $ \bm{x}^1 $ is an optimal solution.
	
	Conversely, suppose 
	\begin{equation}\label{ass0}
	\bm{x}^1 \neq   \bm{x}^\star. 
	\end{equation}	
	Then,  \eqref{rel0} admits
	\begin{align}\label{c1}
	\bm{0}\, = \,    \nabla f(\bm{x}^1) + \text{abs}\big(   \bm{\lambda}^\star \oslash \bm{x}^\star \odot \bm{x}^1\big)  \odot   \text{sgn} (\bm{x}^1) ,
	\end{align}	
	which involves  3 sub-cases.

	 \vspace{5pt}
	 
	$ \bullet $ (i) Additional to \eqref{ass0}, suppose the following  holds:
	\begin{equation}\label{as01}
	\emptyset \neq \{i\, |\,  x^1_i \neq 0, \,  x^\star_i = 0 \} .
	\end{equation}
	Then, for any such index $ i $, we instantly have
	\begin{equation}
	\{\mathcal{M}\bm{x}^1\}_i  = 
\text{abs}\bigg(	\frac{\lambda_i^\star}{x^\star_i }\cdot  x^1_i\bigg)  \cdot \text{sgn} ( x^1_i )\, \rightarrow\, + \infty,
	\end{equation}
	where by  $ \{\cdot\}_i $ we denote the $ i $-th element of the set.
	Then, by \eqref{rel0}, the following always hods:
	\begin{align}
		   0\,& = \,    \{\nabla f(\bm{x}^1)\}_i + \{\mathcal{M}\bm{x}^1\}_i   \nonumber\\
	\iff
	0\, &= \,     \{\nabla f(\bm{x}^1)\}_i +  ( + \infty ),
	\end{align}
	which  says that the gradient is unbounded below. This case is excluded under the basic CCP assumption $ f \in \Gamma_0  (\mathscr H)$.
	
		 \vspace{5pt}
	
	$ \bullet $ (ii) Additional to \eqref{ass0}, suppose the following  holds:
	\begin{equation}
	\emptyset \neq \{j\, |\,  x^1_j = 0, \,  x^\star_j \neq 0 \} .
	\end{equation}
	Then, for any such index $ j $,
		\begin{align}
	\{\mathcal{M}\bm{x}^1\}_j  
	&= \text{abs}\bigg(\frac{\lambda_j^\star}{x^\star_j }\cdot  x^1_j\bigg)  \cdot \text{sgn} ( x^1_j ) \nonumber\\
	&=  0 \nonumber\\
	& \in   \alpha \{\partial \Vert  \bm{x}^1 \Vert_1\}_j
	\end{align}
	Then, by \eqref{rel0}, the following always hods:
      \begin{align}
      0\, &= \,    \{\nabla f(\bm{x}^1)\}_j + \{\mathcal{M}\bm{x}^1\}_j   \nonumber\\
      \iff
      0\, &\in \,     \{\nabla f(\bm{x}^1)\}_j +  \alpha\{\partial \Vert  \bm{x}^1 \Vert_1\}_j,
      \end{align}
	which satisfies \eqref{basic}, and therefore $ \bm{x}^1 $ is an optimal solution. This is a contradiction to our assumption \eqref{ass0}.

	   \vspace{5pt} 	
	   	
	   	$ \bullet $ (iii) Additional to \eqref{ass0}, suppose the following  holds:
	   	\begin{equation}\label{assu_l}
	   \emptyset \neq \{l\, |\,  x^1_l \neq 0, \,  x^\star_l\neq 0 \} .
	   \end{equation}
	   Then, for any such index $ l $, we can denote their scalar difference by $\kappa \neq 0$, i.e.,
	   \begin{equation}
	   x^1_l = \kappa\, x^\star_l.
	   \end{equation}
	  	Then, 	 for any such index $ i $, by Lemma \ref{lem02}, the non-zero element $  x^\star_l\neq 0 $ gives
	  	\begin{equation}
	  	 \lambda^\star_l = \alpha \, \text{sgn} (x^\star_l).
	  	\end{equation}
	  It follows that
	  \begin{align}
	  \{\mathcal{M}\bm{x}^1\}_l  
	  &=  \text{abs}\bigg(	\frac{\lambda_l^\star}{x^\star_l }\bigg)  \cdot  \kappa\, x^\star_l \nonumber\\
	  &=  \kappa\alpha\,  \text{abs}\bigg( \text{sgn} (x^\star_l) \bigg)  \cdot \text{sgn} (x^\star_l)\nonumber\\
	  &=  \kappa\alpha\,  \text{sgn} (x^\star_l)
	  \end{align}
	Then, by \eqref{rel0}, the following always hods:
	   \begin{align}
	   &0\, = \,    \{\nabla f(\bm{x}^1)\}_l + \{\mathcal{M}\bm{x}^1\}_l   \nonumber\\
	   \iff&
	   0\, = \,     \nabla_l f( \kappa\, {x}^\star_l) +  \kappa\alpha \,  \text{sgn} \big(  {x}^\star_l  \big)  \nonumber\\
	    \iff&
	    \text{sgn} \big(  {x}^\star_l  \big)\, = \,    -\frac{1}{\kappa\alpha} \nabla_l f( \kappa\, {x}^\star_l)   ,
	   \end{align}
	   Since $ {x}^\star_l $ is optimal, by the  optimality condition \eqref{basic}, we have
	    \begin{align}
	   0\, = \,     \nabla_l f(  {x}^\star_l) +  \alpha \,  \text{sgn} \big(  {x}^\star_l  \big)  ,
	   \end{align}
	   Combining the above two relations, yields
	   \begin{equation}
	   \nabla_l f( \kappa\, {x}^\star_l)  = \kappa\,\nabla_l f(  {x}^\star_l),
	   \end{equation}
	   which implies $ \nabla_l f $ is linear, and hence $ f_l $ is quadratic, of form
	   \begin{equation}
	   f_l (x_l) \eqdef  {x}_l\cdot\beta\cdot {x}_l,
	   \end{equation}
	   with $ \beta \in\, ]0, +\infty] $ a certain positive scalar.
	   Since $ l_1 $-norm is separable,  problem \eqref{l1_pro} can be considered separately, and its $ l $-th subproblem is 
	\begin{align}
	\underset{{x}_l}{\text{minimize}}\quad   {x}_l\cdot\beta\cdot {x}_l + \alpha  \cdot \text{sgn} (x_l) ,
	\end{align}
	   with optimal solution 
	   \begin{equation}
	    {x}_l^\star = 0.
	   \end{equation}
	   This is a contradiction with our non-zero element assumption, recall from \eqref{assu_l}.
	   
	   In summary, $ \bm{x}^1  $ has to be the optimal solution.
	   The proof is therefore concluded.
	      
   \end{proof}


	\subsection{Extension}
	Here, in a simple way, we  extend our results to the following  more general $ l_1 $-minimization:
	\begin{align}
	\underset{\bm{x}}{\text{minimize}}\quad   f(\bm{x}) + \alpha\Vert \bm{Fx} \Vert_1,
	\end{align}
	with  $ f \in \Gamma_0 (\mathscr H)  $  being smooth, and  $\alpha > 0$ a   regularization parameter, and $ \bm{F} $ a full column-rank (injective) matrix.
	The above can be rewritten into:
	\begin{align}\label{ext}
	&\underset{\bm{x},\bm{z}}{\text{minimize}}\quad   f(\bm{x}) + \alpha\Vert \bm{z} \Vert_1 \nonumber\\
	&\text{subject\,to}\quad  \bm{Fx} = \bm{z}.\tag{$ l_1 $ 2}
	\end{align}
	The corresponding augmented Lagrangian is 
	\begin{equation}
	\mathcal{L}_\mathcal{M}(\bm{x},\bm{z},\bm{\lambda}) 
	\,\eqdef\,  f(\bm{x}) + \alpha \Vert   \bm{z} \Vert_1 + \frac{1}{2}\Vert  \bm{Fx} - \bm{z}  + \mathcal{M}^{-1}\bm{\lambda}\Vert^2_\mathcal{M},  
	\end{equation}
	We have  the following  extended results:

\vspace{10pt}
	$\bullet$ (i) \,\,\textbf{Metric choice extension}:
	The extended Proposition \ref{prop_matrix_step}: simply replace all $ \bm{x}^\star $-related terms  into  $ \bm{Fx}^\star $, or equivalently, $ \bm{z}^\star $.

\vspace{10pt}
	$\bullet$  (ii) \textbf{One-iteration convergence extension}: The extended Claim \ref{claim}:
	
	\begin{claim}\label{claim2}
		Selecting metric $\mathcal{M}$ via the above extension (i). Then, 
		\begin{equation}
		\bm{x}^{1}  =\, \underset{\bm{x}}{\text{argmin}} \, 	\mathcal{L}_\mathcal{M}(\bm{x},0,0)
		=\, \underset{\bm{x}}{\text{argmin}} \,\, f (\bm{x} ) + \frac{1}{2}\Vert\bm{Fx} - 0 \Vert^2_\mathcal{M}. 
		\end{equation}  
		is an optimal solution to 
	\begin{align}\label{ext_p}
&\underset{\bm{x},\bm{z}}{\text{minimize}}\quad   f(\bm{x}) + \alpha\Vert \bm{z} \Vert_1 \nonumber\\
&\text{subject\,to}\quad  \bm{Fx} = \bm{z}.\tag{$ l_1 $ 2}
\end{align}
	\end{claim}

Due to $ \bm{F} $ is assumed being injective, the composition $ f \circ \bm{F}^{-1} $ has a well-defined proximal operator, recall Sec. \ref{sec_inj}. Hence, it is safe to consider the following reformulation:
\begin{align}
&\underset{\bm{z}}{\text{minimize}}\quad   f(\bm{F}^{-1}\bm{z}) + \alpha\Vert \bm{z} \Vert_1 \nonumber\\
&\text{subject\,to}\quad  \bm{x} = \bm{F}^{-1}\bm{z}.
\end{align} 
Then, it can be rewritten into an unconstrained form
\begin{align}
\underset{\bm{z}}{\text{minimize}}\quad   \widetilde{f}(\bm{z}) + \alpha\Vert \bm{z} \Vert_1,
\end{align}
with $  \widetilde{f} \eqdef f\circ \bm{F}^{-1}$. The above shares the same structure as the original \eqref{basic0}. Hence, the same arguments and results apply.

	\section{Conclusion}
	In this paper, a  hidden scaling issue that is  intrinsically  associated  with the classical way of  parametrization is revealed. It is the source of  inconsistencies and complications of some fundamental mathematical  tools, particularly the Moreau identity.  We address this  issue by proposing the \textit{Equilibrate Parametrization}.  A series of useful results  are obtained owing to it. Particularly, for the first time, the general  metric choice that optimizes a worst-case convergence  rate is established.  Also, equivalence is shown between the widely  used preconditioning technique and the metric selection issue. On the application side, a general $ l_1 $-norm regularized problem is studied.  We propose a  closed-form metric choice tailored to it, 
	which yields an optimal solution after a one-time proximal operator evaluation. It is worth noticing that  this ultimate one-iteration convergence requires part of the  optimal  point information (an element-wise ratio  knowledge). The successive estimation technique from \cite{ran2023general} appears not to work well  when applied element-wisely. How to perform a good estimation such  that near  one-iteration convergence can be attained for  practical use is left for future research.

\bibliographystyle{unsrt}
\bibliography{Reference/ref1,Reference/Ref_S,Reference/ML_application,Reference/sr_applications}

\end{document}